\documentclass[10pt]{article}
\usepackage{lmodern}
\usepackage{amsmath}
\usepackage[T1]{fontenc}
\usepackage[utf8]{inputenc}
\usepackage{authblk}
\usepackage{amsfonts}
\usepackage{graphicx}
\usepackage{rotating}
\usepackage{amssymb}
\usepackage[english]{babel}
\usepackage{xcolor}
\usepackage{amsthm}
\usepackage{graphicx}
\usepackage{mathrsfs}
\usepackage{makecell}
\usepackage{microtype}
\usepackage{mathscinet}
\usepackage{array}
\usepackage{multirow}
\usepackage{enumerate}
\usepackage[cal=boondoxo,bb=ams]{mathalfa}
\usepackage{hyperref}
\hypersetup{hidelinks}

\newtheorem{defn}{Definition}[section]
\newtheorem{theorem}{Theorem}[section]
\newtheorem{prop}{Proposition}[section]
\newtheorem{lemma}{Lemma}[section]

\newtheorem{remark}{Remark}[section]
\newtheorem{exam}{Example}[section]

\newcommand{\ml}{\mathcal}
\newcommand{\mb}{\mathbb}
\DeclareMathOperator{\lin}{lin}
\DeclareMathOperator{\nlin}{nlin}

\def\XXint#1#2#3{{\setbox0=\hbox{$#1{#2#3}{\int}$ }
		\vcenter{\hbox{$#2#3$ }}\kern-.6\wd0}}

\title{Sharp lifespan estimates for semilinear fractional evolution equations with critical nonlinearity}
\author[1]{Wenhui Chen\thanks{Wenhui Chen (wenhui.chen.math@gmail.com)}}
\affil[1]{School of Mathematics and Information Science, Guangzhou University,\authorcr 510006 Guangzhou, China}

\author[2]{Giovanni Girardi\thanks{Giovanni Girardi (g.girardi@univpm.it)}}
\affil[2]{Department of Industrial Engineering and Mathematical Sciences, \authorcr Polytechnic University of Marche,  60131 Ancona, Italy}
\date{}

\begin{document}
		\maketitle

		\begin{abstract}
			\medskip
In this paper we consider semilinear wave equation and other second order $\sigma$-evolution equations with different (effective or non-effective) damping mechanisms driven by fractional Laplace operators; in particular, the nonlinear term is the product of a power nonlinearity $|u|^p$ with the critical exponent $p=p_{\mathrm{c}}(n)$ and a modulus of continuity $\mu(|u|)$. We derive a critical condition on the nonlinearity by proving a global in time existence result under the Dini condition on $\mu$ and a blow-up result when $\mu$ does not satisfy the Dini condition. Especially, in this latter case we determine new sharp estimates for the lifespan of local solutions, obtaining coincident upper and lower bounds of the lifespan. In particular, we derive a new sharp estimate for the wave equation with structural damping and classical power nonlinearity $|u|^p$ in the critical case $p=p_c(n)$, not yet determined in previous literature. The proof of the blow-up results and the upper bound estimates of the lifespan require the introduction of new test functions which allows to overcome some new difficulties due to the presence of both non-local differential operators and general nonlinearities. 
			\\
			
			
			\noindent\textbf{AMS Classification (2020)}  35L15, 35G20, 35R11, 35B33, 35A01, 35B44
		\end{abstract}
\fontsize{12}{15}
\selectfont
%


\section{Introduction}\label{Section_Introduction}
\hspace{5mm}In the present paper we consider the Cauchy problem for the following nonlinear evolution equation:
\begin{equation}
	\label{eq:CPgeneral}
	u_{tt}+(-\Delta)^{\sigma}u+(-\Delta)^{\delta}u_t=|u|^p\mu(|u|), \quad x\in\mb{R}^n,\ t\geqslant 0,\\
\end{equation}
with $\sigma\geqslant 1$, $\delta\in[0,\sigma]$, $p>1$ and $\mu$ a modulus of continuity. Proceeding in the same direction of previous works (see, for instance, \cite{D'Abbicco-Girardi=2023,Ebert-Girardi-Reissig=2020,Girardi=2024}) our aim is to improve the known results on the existence of local or global in time solutions to the Cauchy problem associated to \eqref{eq:CPgeneral} depending on whether $\mu$ satisfies the Dini condition or not (see later Definition \ref{Defn-Dini}), even in the case of fractional values of $\sigma$ and $\delta$; in particular, we will provide the sharp lifespan estimates for the Cauchy problem \eqref{eq:CPgeneral} which depend on the choice of $p>1$ and $\mu=\mu(s)$ (see later Theorems \ref{Thm-Lower-Bound} and \ref{Thm-Upper-Bound}). Before stating additional details about the main purposes of this manuscript, in the following we provide a brief review about existing results for model \eqref{eq:CPgeneral}, first with the classical power nonlinearity $|u|^p$, and then with the general nonlinearity $|u|^p\mu(|u|)$; in particular, we already provide a new sharp lifespan estimate for the model \eqref{eq:CPgeneral} in the case $\mu\equiv 1$. Throughout all the paper, we use the following notations.
\subsection{Notation}
\hspace{5mm}The constants $c$ and $C$ may be different from line to line but are always independent of the parameter $\varepsilon$. We write $f\lesssim h$ if there exists a positive constant $C$ such that $f\leqslant Ch$; analogously for $f\gtrsim h$. The sharp relation $f\approx h$ holds if, and only if, $h\lesssim f\lesssim h$. For any $a\in\mb{R}$ we consider its floor $[a]:=\max\{\tilde{a}\in\mb{Z}:\tilde{a}\leqslant a \}$ and its fractional part $\{a\}:=a-[a]$. The symbol $\langle x\rangle^2:=1+|x|^2 $ denotes the Japanese bracket.  For any $s\in\mb{R}$ and $m\in[1,+\infty)$ we consider the Bessel potential space $H^s_m=H^s_m(\mb{R}^n)$. The fractional differential operator $|D|^{s}$ is defined by its action $|D|^sf:=\mathcal{F}^{-1}(|\xi|^s\widehat{f}\,)$  where $ \mathcal{F}$ is the Fourier transform with respect to the spatial variable,  and in turn $\ml{F}^{-1}$ denotes its inverse so that $\widehat{f}=\mathcal{F}(f)$; moreover,  $(-\Delta)^{s}:H^{2s}\to L^2$ is defined by $(-\Delta)^{s}h:=\ml{F}^{-1}(|\xi|^{2s}\widehat{h})$. Especially, if $s\in\mb{R}_+\backslash\mb{N}_+$ the non-local operator $(-\Delta)^{s}$ can be also defined via the integral representation (cf. \cite[Section 3]{Di-Pal-Vald=2012})  
\begin{align}\label{Frac-Defn}
	(-\Delta)^{s}h(x):=(-1)^{[s]+1}C_{2s}\int_{\mb{R}^n}\frac{(\tau_{y/2}-\tau_{-y/2})^{2[s]+2}h(x)}{|y|^{n+2s}}\,\mathrm{d}y
\end{align}
(here, $[\cdot]$ denotes the floor function) for any $h\in H^{2s}$, with the translation operator $\tau_yh(x):=h(x+y)$ and the positive constant
\begin{align*}
	C_{2s}:=2^{-2[s]-2+2s}\left(\int_{\mb{R}^n}\frac{(\sin y_1)^{2[s]+2}}{|y|^{n+2s}}\,\mathrm{d}y\right)^{-1} \ \ \mbox{where}\ \ y=(y_1,\dots,y_n).
\end{align*}
For any $p\in [1,+\infty]$ the $h(x)$-weighted space $L^p(\mb{R}^n;h(x)\,\mathrm{d}x)$ corresponds to the  space $L^p=L^p(\mb{R}^n)$ with the weighted measure $h(x)\,\mathrm{d}x$.
For any $k\in\mathbb{N}_+$ we denote by $\ln^{[k]}$ and $\exp^{[k]}$ the $k$-times iterated logarithmic and exponential functions, respectively,
\begin{align*}
	& \ln^{[k]} := \begin{cases}
		\ln & \mbox{if} \ \ k=1, \\
		\ln \circ \ln^{[k-1]}  & \mbox{if} \ \ k\geqslant 2, 
	\end{cases} 
	\ \ \mbox{and}\ \  \exp^{[k]} := \begin{cases}
		\exp & \mbox{if} \ \ k=1, \\
		\exp \circ \exp^{[k-1]}  & \mbox{if} \ \ k\geqslant 2.
	\end{cases}
\end{align*}
Finally, given $\sigma\geqslant 1$ and $\delta \in [0,\sigma]$ we denote by $s_0$ and $q_0$ the following quantities, which play a fundamental role in the proof of non-existence results:
\begin{align}
s_0:=\begin{cases}
\min\{\sigma-[\sigma],\delta-\delta\}&\mbox{if}\ \ \sigma,\delta\not\in \mb{N}_0,\\
\sigma-[\sigma]&\mbox{if}\ \ \sigma\not\in\mb{N}_0,\,\delta\in\mb{N}_0,\\
\delta-[\delta]&\mbox{if}\ \ \sigma\in\mb{N}_0,\,\delta\not\in\mb{N}_0,\\
\epsilon_0&\mbox{if}\ \ \sigma,\delta\in\mb{N}_0, \end{cases}\qquad
q_0:=n+2s_0\label{eq:s0}
\end{align}
where $\epsilon_0>0$ is a fixed arbitrary constant.
\subsection{Power-type nonlinearity in semilinear evolution equations}
\hspace{5mm}In recent years, the following Cauchy problem for the semilinear evolution equations with different damping mechanisms and a power-type nonlinearity has been deeply studied:
\begin{align}\label{Sigma-Evolution-Power}
	\begin{cases}
		u_{tt}+(-\Delta)^{\sigma}u+(-\Delta)^{\delta}u_t=|u|^{p},&x\in\mb{R}^n,\ t>0,\\
		u(0,x)=\varepsilon u_0(x),\ u_t(0,x)=\varepsilon u_1(x),&x\in\mb{R}^n,
	\end{cases}
\end{align}
with $\sigma\geqslant 1$, $\delta\in[0,\sigma]$ and $p>1$, where $\varepsilon>0$ is a parameter describing the size of the Cauchy data. In particular, $\delta=0$ in the model corresponds to a friction or an external damping, whereas the cases $\delta\in(0,\sigma/2]$ and $\delta\in(\sigma/2,\sigma]$ are referred to as the structural effective damping and,  respectively,  the structural non-effective damping. 

It is known that, assuming $L^1$ regularity for the initial data, the critical exponent for the semilinear Cauchy problem \eqref{Sigma-Evolution-Power} is given by 
\begin{align}\label{Critical-Exponent}
	p_{\mathrm{c}}(n):=1+\frac{2\sigma}{n-\min\{2\delta,\sigma\}}, \ \ \mbox{when}\ \ n>\min\{2\delta,\sigma\};
\end{align}
in the space dimension $n\leqslant \min\{2\delta,\sigma\}$ it is customary to set $p_{\mathrm{c}}(n):=+\infty$, since the semilinear Cauchy problem \eqref{Sigma-Evolution-Power} does not admit any global in time weak solution for all $p>1$. Here, by \textit{critical exponent} we mean a threshold value in the range of power nonlinearities $|u|^p$ such that for  $1<p<p_{\mathrm{c}}(n)$, i.e. in the sub-critical case, local in time solutions blow up in finite time under suitable sign assumptions for the Cauchy data, regardless of their sizes; whereas for $p>p_{\mathrm{c}}(n)$, i.e. in the super-critical case, there exist solutions which are globally in time defined in suitable function spaces provided that the Cauchy data are sufficiently small; the critical case $p=p_{\mathrm{c}}(n)$ belongs to the blow-up range as well (note that the critical exponent is shifted if the initial data are not in $L^1$, but in $L^m$ for some $m>1$; see,  for example, \cite{Chen-D'Abbicco-Girardi=2022, Ikeda-Inui-Okamoto-Wakasugi=2019}).

Under different conditions on $\sigma$ and $\delta$ the critical exponent \eqref{Critical-Exponent} is justified by the results in \cite{D'Abbicco-Ebert=2014, D'Abbicco-Ebert=2017, D'Abbicco-Ebert=2022, D'Abbicco-Fujiwara=2021,  D'Abbicco-Reissig=2014, Dao-Reissig=2021-01, Pham-K-Reissig=2015} and references given therein. 
We underline that in the non-effective case $\delta\in(\sigma/2,\sigma]$ the formula \eqref{Critical-Exponent} determines the critical exponent if $\sigma>1$ only; indeed, determining the critical exponent for $\sigma=1$ and, simultaneously, $\delta\in(1/2,1]$ is still an open problem, although sharp $L^p-L^q$ estimates in this case have been derived in \cite{D'Abbicco-Ebert=2025} already (see also \cite{D'Abbicco-Lagioia=2025} where the authors obtained decay estimates for the wave equation with a very strong damping, namely in the case  $\sigma=1$ and $\delta >1$).

Since no global in time solutions to \eqref{Sigma-Evolution-Power} exist for every $1<p\leqslant p_{\mathrm{c}}(n)$ and $\varepsilon>0$, it is interesting to investigate lifespan estimates for the local solutions.  Let us recall that, fixed a parameter $\varepsilon\in (0,1)$ which determines the size of initial data, the lifespan $T_\varepsilon$ of a given Cauchy problem is defined as the maximal existence time in which the Cauchy problem admits a solution, depending on the value of $\varepsilon$. In particular, we define $T_{\varepsilon}=+\infty$ if a (small data) global in time solution exists. 

%
In the special case $(\sigma,\delta)=(1,0)$, that corresponds to the semilinear classical damped wave equation (see   \cite{Fujiwara-Georgiev=2023,Fujiwara-Georgiev=2024, Fujiwara-Ikeda-Wakasugi=2019, Ikeda-Ogawa=2016, Ikeda-Sobajima=2019,Kirane-Qafsaoui=2002, Lai-Zhou=2019, Li-Zhou=1995}), the lifespan  is sharply estimated by 
\begin{align}\label{Sharp-Damped-Waves}
	T_{\varepsilon}\approx\begin{cases}
		\varepsilon^{-\frac{2(p-1)}{n(p_{\mathrm{F}}(n)-p)}}&\mbox{if}\ \ 1<p<p_{\mathrm{F}}(n),\\
		\exp\left(C\varepsilon^{-(p_{\mathrm{F}}(n)-1)}\right)&\mbox{if}\ \ p=p_{\mathrm{F}}(n),
	\end{cases}
\end{align}
and $T_{\varepsilon}=+\infty$ if $p>p_{\mathrm{F}}(n)$.  In \eqref{Sharp-Damped-Waves} by $p_{\mathrm{F}}(n)$ we denote the Fujita exponent $p_{\mathrm{F}}(n):=1+2/n$; it was introduced in \cite{Fujita=1966} for the semilinear heat equation, and coincides with the critical exponent for the semilinear damped wave equation $u_{tt}-\Delta u+u_t=|u|^p$ due to the well-known diffusion phenomenon (we refer the interested reader to \cite{Ikeh-Tani-2005, Li-Zhou=1995,Todorova-Yordanov=2001,Zhang=2001} and references therein for detailed explanation). 
We stress that the lifespan estimate \eqref{Sharp-Damped-Waves} is sharp in the sense that the upper and lower bounds of the lifespan $T_\varepsilon$ are exactly coincident except for some multiplicative constants independent of $\varepsilon$ (the notation $\approx$ is always used for denoting a sharp estimate).

For $\delta\leqslant\sigma\neq 1$ both integers, the authors in \cite[Theorem 1]{D'Abbicco-Ebert=2017} derived an upper bound estimate of the lifespan only in the sub-critical case; 
more precisely, when $\sigma\in\mb{N}_+$ and $\delta\leqslant\sigma$ belongs to $\mb{N}_0$, assuming $u_0=0$ and a suitable sign assumption on $u_1\in L^1$, they obtained
\begin{align}\label{Upper-Bound-D'Abb-Ebert-17}
	T_{\varepsilon}\lesssim\varepsilon^{-\frac{\kappa}{2\sigma p'-(n+\kappa)}}= \varepsilon^{-\frac{\kappa(p-1)}{(n-\min\{2\delta,\sigma\})(p_{\mathrm{c}}(n)-p)}}\ \ \mbox{if}\ \ 1<p<p_{\mathrm{c}}(n),
\end{align}
for every $\varepsilon\in(0,1)$. Here, $p'=p/(p-1)$ is the H\"older's conjugate of $p$ and, for the sake of simplicity, $\kappa$ is defined by
\begin{align*}
	\kappa:=2\sigma-\min\{2\delta,\sigma\}.
\end{align*}
To the best of our knowledge, the sharpness of estimate \eqref{Upper-Bound-D'Abb-Ebert-17} has not been proved in the literature; moreover, (sharp) lifespan estimates in the critical case $p=p_{\mathrm{c}}(n)$ are in general unknown except in the special case $(\sigma, \delta)=(1,0)$.
Following the same approach used in the proof of our main results Theorem \ref{Thm-Lower-Bound} and Theorem \ref{Thm-Upper-Bound}, one can prove that the estimate \eqref{Upper-Bound-D'Abb-Ebert-17} is sharp; particularly, in the space dimension $n>\min\{2\delta, \sigma\}$ it holds true for every $1<p< p_{\mathrm{c}}(n)$ even for not integer values of both $\sigma\neq 1$ and $\delta\leqslant\sigma$. Furthermore, a new sharp lifespan estimate in the critical case $p=p_{\mathrm{c}}(n)$ can be derived. More in detail, the following proposition can be proved.

\begin{prop}\label{Coro-Classical-sigma-evolution}
	Let $\sigma\geqslant 1$, $\delta\in[0,\sigma]$ and $q_0$ defined by \eqref{eq:s0}; we assume that the Cauchy data $u_0$ and $u_1$ satisfy
	\begin{align*}
		u_0,\, u_1\in L^1(\mb R^n; \langle x\rangle^{-q_0}\,\mathrm{d}x) \quad \text{and} \quad \int_{\mb R^n} \big(u_0(x)+u_1(x)\big)\langle x\rangle^{-q_0}\,\mathrm{d}x >0,
	\end{align*}
	if $\delta=0$; whereas if $\delta>0$ we suppose that $u_1$ belongs to $L^1(\mb R^n; \langle x\rangle^{-q_0}\,\mathrm{d}x)$, and additionally one of the following conditions holds:
	\begin{subequations}
		\begin{align*}
			u_0\in L^1(\mb R^n) \quad &\text{and} \quad  \int_{\mb R^n}u_1(x)\langle x\rangle^{-q_0}\,\mathrm{d}x >0,\\
			u_0\in L^1(\mb R^n; \langle x\rangle^{-q_0}\,\mathrm{d}x) \quad &\text{and} \quad \exists \bar{c}>0 : u_1(x)\geqslant 2\bar{c}\, |u_0(x)| \quad \forall x\in \mb R^n.
		\end{align*}
	\end{subequations}
	Then, the lifespan $T_{\varepsilon}$ for the semilinear evolution equations \eqref{Sigma-Evolution-Power} with power nonlinearity $|u|^p$ satisfies the following sharp estimates:
	\begin{itemize}
		\item in the sub-critical case $1<p<p_{\mathrm{c}}(n)$,
		\begin{align}\label{Lifespan-Example-Sub-Critical}
			T_{\varepsilon}\approx \varepsilon^{-\frac{\kappa(p-1)}{(n-\min\{2\delta,\sigma\})(p_{\mathrm{c}}(n)-p)}};
		\end{align}
		\item in the critical case $p=p_{\mathrm{c}}(n)$,
		\begin{align}\label{Lifespan-Example-Critical}
			T_{\varepsilon}\approx \exp\left(C\varepsilon^{-\frac{2\sigma}{n-\min\{2\delta,\sigma\}}}\right).
		\end{align}
	\end{itemize}
\end{prop}
\begin{proof}
	The proof of the desired lifespan estimates \eqref{Lifespan-Example-Sub-Critical} and \eqref{Lifespan-Example-Critical} follows by straightforward calculations as those in the proof of our main results Theorem \ref{Thm-Lower-Bound} and Theorem \ref{Thm-Upper-Bound}, fixing $\mu(\tau)= \tau^{p-p_{\mathrm{c}}(n)}$ in the sub-critical case $1<p<p_{\mathrm{c}}(n)$ and $\mu\equiv 1$ in the critical case $p=p_{\mathrm{c}}(n)$.
	In particular, the estimates  \eqref{LLLL-01} and \eqref{Contradiction-01} can be proved, with
	\begin{align*}
		\ml{H}(\tau)=\frac{1}{p_{\mathrm{c}}(n)-p}\left(\tau^{p-p_{\mathrm{c}}(n)}-\tau^{p-p_{\mathrm{c}}(n)}_0\right)
	\end{align*}
	in the sub-critical case $1<p<p_{\mathrm{c}}(n)$ and $\ml{H}(\tau)=\ln \tau_0-\ln\tau$  in the critical case $p=p_{\mathrm{c}}(n)$; then, the proof follows the strategy explained in Sections \ref{Subsection-Philosophy} and \ref{Section-Blow-up_1}.
\end{proof}
\begin{remark}
	The estimate \eqref{Lifespan-Example-Sub-Critical} with $(\sigma,\delta)=(1,0)$ exactly coincides with the sharp lifespan estimate \eqref{Sharp-Damped-Waves} for the semilinear damped wave equation $u_{tt}-\Delta u+u_t=|u|^p$ in the sub-critical Fujita case $1<p<p_{\mathrm{F}}(n)$; in the general situation $\delta\in[0,\sigma]$, the estimate \eqref{Lifespan-Example-Sub-Critical} verifies the sharpness of the upper bound estimate \eqref{Upper-Bound-D'Abb-Ebert-17} derived in \cite[Theorem 1]{D'Abbicco-Ebert=2017}. On the other hand, in the critical case $p=p_{\mathrm{c}}(n)$, the estimate \eqref{Lifespan-Example-Critical} with $(\sigma,\delta)=(1,0)$ exactly coincides with the sharp lifespan estimate \eqref{Sharp-Damped-Waves} for the critical semilinear damped wave equation $u_{tt}-\Delta u+u_t=|u|^{p_{\mathrm{F}}(n)}$; in the general case $\delta\in[0,\sigma]$, \eqref{Lifespan-Example-Critical} is a new sharp estimate of the lifespan.
\end{remark}

\subsection{Modulus of continuity in critical semilinear evolution equations}
\hspace{5mm}We can rephrase the critical nature of the exponent \eqref{Critical-Exponent} by saying that $p=p_{\mathrm{c}}(n)$ is the threshold value that separates the blow-up range from the small data solutions' global existence range in the scale of power-type nonlinearities $\{|u|^p:p>1\}$. Following the philosophy of \cite{Ebert-Girardi-Reissig=2020} one may investigate how blow-up results and lifespan estimates are affected if one considers a more refined scale of nonlinearities, in the form $\{|u|^{p_{\mathrm{c}}(n)}\mu(|u|):\mu\mbox{ is a modulus of continuity}\}$; we will show that the Dini condition on $\mu$ (see Definition \ref{Defn-Dini}) will play a crucial role in this study.

Let us recall the notion of Dini condition introduced in the classical work by Ulisse Dini \cite[Page 101]{Dini=1880} (see also \cite{Zygmund=1988}) to investigate the convergence of some Fourier series, and later applied in different areas of mathematical analysis, e.g. elliptic equations \cite{Hartman-Wintner=1955}, dynamical systems \cite{Fan-Jiang=2001}, harmonic analysis \cite{Lerner-Omb-Riv=2017}, complex analysis \cite{BADGER}, probability theory \cite{Bass=1988}, etc. 

\begin{defn}\label{Defn-Dini}
	Let $\mu:[0,+\infty)\to[0,+\infty)$ be a modulus of continuity, that is, $\mu$ is a continuous and increasing function such that $\mu(0)=0$. Then, $\mu$ satisfies the Dini condition if
	\begin{align}\label{Dini-Condition}
		\int_0^{\tau_0}\frac{\mu(\tau)}{\tau}\,\mathrm{d}\tau<+\infty,\ \ \mbox{for some}\ \ \tau_0>0. 
	\end{align}
	On the other hand, $\mu$ satisfies the non-Dini condition if \eqref{Dini-Condition} does not hold, i.e., if
	\begin{align}\label{not-Dini-Condition}
		\int_0^{\tau_0}\frac{\mu(\tau)}{\tau}\,\mathrm{d}\tau=+\infty,\ \ \mbox{for every}\ \ \tau_0>0.
	\end{align}
\end{defn}
\begin{remark}
	If the integral condition \eqref{Dini-Condition} is satisfied for some $\tau_0>0$, then it also holds true for any $\tau_0>0$ due to the local summability of 
	$\mu(\tau)/\tau$ in any compact set of $\mb{R}_+$.
\end{remark}

In 2020, the authors of \cite{Ebert-Girardi-Reissig=2020} studied the following semilinear damped wave equation with a modulus of continuity $\mu=\mu(s)$ in the nonlinearity:
\begin{align}\label{Damped-Wave-Modulus}
	\begin{cases}
		u_{tt}-\Delta u+u_t=|u|^{p_{\mathrm{F}}(n)}\mu(|u|),&x\in\mb{R}^n,\ t>0,\\
		u(0,x)=\varepsilon u_0(x),\ u_t(0,x)=\varepsilon u_1(x),&x\in\mb{R}^n;
	\end{cases}
\end{align}
it is the first paper in which the problem of determining a critical condition on nonlinearities in the form $|u|^p\mu(|u|)$ is considered, with the aim of refining the well-known results about the existence of global in time solutions to the semilinear damped wave equation in $\mb R^n$. The authors proved that the Dini condition \eqref{Dini-Condition} allows to identify the desired critical condition on $\mu$, enabling to distinguish more precisely the region of existence of a global in time small data solution from that in which the problem admits no global in time weak solutions; especially, in the space dimension $n=1,2$ they applied Matsumura's estimates and a fixed point argument to prove that the Dini condition \eqref{Dini-Condition} on $\mu$ is sufficient to guarantee the global in time existence of small data solutions to \eqref{Damped-Wave-Modulus}; on the other hand, for any $n\geqslant 1$ they applied a modified test function method, introduced in \cite{Ikeda-Sobajima=2019}, to prove that \eqref{Dini-Condition} is also a necessary condition for the existence of global in time weak solutions to \eqref{Damped-Wave-Modulus}. From these results  one can expect that the classical scale $\{|u|^p:p>1\}$ is in general too rough to determine the critical nonlinearity for semilinear evolution models.
Moreover, if $\mu$ does not satisfy the Dini condition, it is interesting to determine sharp lifespan estimates (from above and below), which are currently unknown for the problem \eqref{Damped-Wave-Modulus}. In the present manuscript we will answer to this question, providing the sharp lifespan estimates \eqref{Sharp-Lifespan-Est} with $(\sigma,\delta)=(1,0)$.

Soon after, the same problem was investigated in  \cite{Dao-Reissig=2021} for a weakly coupled system of semilinear damped wave equations, with nonlinearities in the form $|s|^p \mu(|s|)$; for the same problem, an example of upper bound estimate of the lifespan was proposed in \cite[Remark 9]{Chen-Dao=2023}, by fixing  $\mu(s)=(\ln(-s))^{-\gamma}$ with $\gamma$ properly chosen. To the best of our knowledge, there are no results on critical nonlinearities in this form for weakly coupled systems of structurally damped wave equations, for which a new sharp condition for the existence of global in time small data solutions have been obtained in \cite{D'Abbicco-LagioiaNODEA=2025} in the non-effective damped case with power nonlinearity. In \cite{M-Djaouti-Reissig=2023} the authors obtained the critical condition \eqref{Dini-Condition}, in the 1-dimensional case, for semilinear wave equations with either a time-dependent effective or a scale-invariant damping term. In \cite{Girardi=2024} the author showed that the results obtained in \cite{Ebert-Girardi-Reissig=2020} for the problem \eqref{Damped-Wave-Modulus} can be generalized to more general semilinear evolution models with the Fujita-type critical exponent; in particular, the critical behavior of nonlinearity is still described by the Dini condition \eqref{Dini-Condition}. 

Quite recently, the critical regularity of nonlinearities for semilinear classical wave equations (without any linear damping mechanisms) was studied in \cite{Chen-Reissig=2023,Wang-Zhang=2024} for the nonlinearity $|u|^{p_{\mathrm{S}}(n)}\mu(|u|)$, and in \cite{Chen-Palmieri=2024} for the nonlinearity $|u_t|^{p_{\mathrm{G}}(n)}\mu(|u_t|)$; here, $p_{\mathrm{S}}(n)$ and $p_{\mathrm{G}}(n)$ stand for the Strauss exponent, and respectively, the Glassey exponent, which are the critical exponents for the semilinear classical wave equations with the power-type nonlinearity $|u|^p$, and respectively, with the derivative-type nonlinearity $|u_t|^p$. In \cite{Chen-Palmieri=2024} sharp lifespan estimates (at least in the radially symmetric and 3-dimensional case) for the semilinear wave model $u_{tt}-\Delta u=|u_t|^{p_{\mathrm{G}}(n)}\mu(|u_t|)$ were derived.

Let us go back to our semilinear evolution model \eqref{eq:CPgeneral}; strongly motivated by \cite{Ebert-Girardi-Reissig=2020} the authors in \cite{D'Abbicco-Girardi=2023,Girardi=2024} considered the refined scale of nonlinear terms $\{|u|^{p_{\mathrm{c}}(n)}\mu(|u|):\mu\mbox{ is a modulus of continuity}\}$ in the following semilinear Cauchy problem:
\begin{align}\label{Sigma-Evolution-Modulus}
	\begin{cases}
		u_{tt}+(-\Delta)^{\sigma}u+(-\Delta)^{\delta}u_t=|u|^{p_{\mathrm{c}}(n)}\mu(|u|),&x\in\mb{R}^n,\ t>0,\\
		u(0,x)=\varepsilon u_0(x),\ u_t(0,x)=\varepsilon u_1(x),&x\in\mb{R}^n,
	\end{cases}
\end{align}
with $p_{\mathrm{c}}(n)$ defined by \eqref{Critical-Exponent}.
Assuming the Dini condition \eqref{Dini-Condition} on $\mu$, in the non-effective case $\delta\in(\sigma/2,\sigma]$ with $\max\{1,n/2\}<\sigma<n$, the authors of \cite{D'Abbicco-Girardi=2023} proved the existence of a global in time solution to \eqref{Sigma-Evolution-Modulus} in the classical Sobolev space $H^{\sigma}$, taking small initial data in $L^2\cap L^1$.
Moreover, for any $\sigma\geqslant 1$, in the limit case $\delta=\sigma/2$ they proved the existence of a global in time solution to \eqref{Sigma-Evolution-Modulus} in $H^{\sigma}\cap L^1\cap L^{\infty}$, assuming small initial data in $L^1\cap L^{\frac{n}{\sigma}}$; as a consequence, the Dini condition \eqref{Dini-Condition} on $\mu$ turns out to be again sufficient to guarantee the existence of global in time small data solutions to \eqref{Sigma-Evolution-Modulus} in the non-effective case $\delta \in [\sigma/2, \sigma]$; the counterpart of non-existence under the non-Dini condition \eqref{not-Dini-Condition} was proved in the recent manuscript \cite{Girardi=2024} for $\delta$ and $\sigma$ both integers. Still in \cite{Girardi=2024}, among other more general differential equations, the author considered the problem \eqref{Sigma-Evolution-Modulus} in the effective case  $\delta\in[0,\sigma/2)$ with $\sigma\in\mb{N}$ and  $\delta\in\mb{N}_0$;  under suitable sign assumptions on the initial data, if $\mu$ satisfies the non-Dini condition \eqref{not-Dini-Condition} then every weak solution blows up in finite time; on the other hand, in the space dimension $n\in(2\delta,2\sigma)$, if $\mu$ satisfies the Dini condition \eqref{Dini-Condition}, then any local in time solution to \eqref{Sigma-Evolution-Modulus} in $L^{p_{\mathrm{c}}(n)}\cap L^{\infty}$ can be globally extended, provided that the initial data are sufficiently small in $(H^1\cap L^1)\times (L^2\cap L^1)$. 

In the fractional case $\sigma,\delta\in\mb{R}_+\backslash\mb{N}_+$, some technical difficulties arise due to the non-local behavior of the differential operators $(-\Delta)^\sigma$ and $(-\Delta)^\delta$; as a consequence, the test function method employed in \cite{Girardi=2024} (or in the previous papers \cite{Ebert-Girardi-Reissig=2020, Ikeda-Sobajima=2019}) does not work anymore.

\subsection{Main purposes of this manuscript}
\hspace{5mm}If $\mu$ satisfies the non-Dini condition \eqref{not-Dini-Condition} the above-mentioned results for the semilinear evolution models \eqref{Damped-Wave-Modulus} and \eqref{Sigma-Evolution-Modulus} show that every non-trivial local in time weak solution blows up in finite time, provided that the Cauchy data satisfy suitable sign assumptions (at least when all the powers of Laplace operators are integers). For this reason, in the present paper we will discuss the following natural and important question:
\begin{center}
	\emph{Is it possible to obtain a sharp  lifespan estimate for our problem \eqref{Sigma-Evolution-Modulus}?}
\end{center}
We will give a positive answer to this question, providing sharp lifespan estimates for the Cauchy problem \eqref{Sigma-Evolution-Modulus}, even in the case of fractional Laplacians, for any $\delta\in[0,\sigma]$. 
More specifically, strongly motivated by the Dini condition \eqref{Dini-Condition}, we introduce the non-negative and decreasing auxiliary function $\ml{H}=\ml{H}(\tau)$ defined by
\begin{align}\label{H-function}
	\ml{H}:\tau\in[0,\tau_0]\to \ml{H}(\tau):=\int_{\tau}^{\tau_0}\frac{\mu(\varrho)}{\varrho}\,\mathrm{d}\varrho \,\in[0,+\infty]
\end{align}
with $\tau_0>0$ fixed. 
\begin{remark}
	Notice that $\ml{H}(0)<+\infty$ if, and only if, the Dini condition \eqref{Dini-Condition} is satisfied; moreover, its inverse function $\ml{H}^{-1}(\omega)$ exists and is unique due to the strict monotonicity of $\ml H$.
\end{remark}
\noindent Recalling $\kappa=2\sigma-\min\{2\delta,\sigma\}$ we are going to determine the sharp lifespan estimates
\begin{align}\label{Sharp-Lifespan-Est}
	T_{\varepsilon}\begin{cases}
		=+\infty&\mbox{if}\ \ \ml{H}(0)<+\infty,\\
		\approx \displaystyle{\left[\varepsilon^{-1}\ml{H}^{-1}\left(C\varepsilon^{-\frac{2\sigma}{n-\min\{2\delta,\sigma\}}}+\ml{H}(c\,\varepsilon)\right)\right]^{-\frac{\kappa}{n-\min\{2\delta,\sigma\}}}}&\mbox{if}\ \ \ml{H}(0)=+\infty,
	\end{cases}
\end{align}
where $c$ and $C$ are positive constants independent of $\varepsilon$. In particular, taking $\mu(\tau)=\tau^{p-p_{\mathrm{c}}(n)}$ and $\mu(\tau)=1$, respectively, one can  derive the sharp lifespan estimates for the semilinear model \eqref{Sigma-Evolution-Power}  with the power-type nonlinearity $|u|^p$ in the sub-critical case $1<p<p_{\mathrm{c}}(n)$ and the critical case $p=p_{\mathrm{c}}(n)$ (see Proposition \ref{Coro-Classical-sigma-evolution}). 

Our paper is organized as follows. In Section \ref{Section-Main-Result} we state our main results, providing the sharp lifespan estimates for the semilinear evolution equation \eqref{Sigma-Evolution-Modulus} even in the case of fractional Laplacians; moreover, we present some typical examples of moduli of continuity which satisfy the assumptions of the given results, showing the corresponding lifespan estimates for the problem \eqref{Sigma-Evolution-Modulus}. On the one hand, motivated by \cite{D'Abbicco-Girardi=2023,Girardi=2024}, in Section \ref{Section-Existence} we apply some known $L^p$-$L^q$ estimates for the linear problem associated to \eqref{Sigma-Evolution-Modulus}, together with the Banach fixed point argument, to demonstrate the existence of local and global $L^{p_{\mathrm{c}}(n)}\cap L^{\infty}$ solutions to \eqref{Sigma-Evolution-Modulus} depending on whether $\mu$ satisfies the Dini condition \eqref{Dini-Condition}; additionally, we derive the lower bound estimate of the lifespan $T_{\varepsilon}$ in \eqref{Sharp-Lifespan-Est} if $\ml{H}(0)=+\infty$. On the other hand, motivated by \cite{D'Abbicco-Fujiwara=2021,Ikeda-Sobajima=2019},  in Section \ref{Section-Blow-up} we develop a modified test function method to prove the blow-up in finite time of any weak solution to \eqref{Sigma-Evolution-Modulus} (in the sense of Definition \ref{Defn-Weak-Solution}) when the non-Dini condition \eqref{not-Dini-Condition} holds, under a suitable sign assumption on the initial data. Particularly, for $\delta$ and $\sigma$ fractional numbers, the upper bound estimate of the lifespan $T_{\varepsilon}$ in \eqref{Sharp-Lifespan-Est} is derived if $\ml{H}(0)=+\infty$, by introducing new test functions which allow to overcome some new difficulties due to the presence of both non-local differential operators and general nonlinearities. We believe that the approaches used in this paper can be applied to derive sharp lifespan estimates in other related models; for instance, the semilinear evolution equation \eqref{eq:CPgeneral} with fractional Laplacians and the derivative-type nonlinearity $|u_t|^{1+\frac{\min\{2\delta,\sigma\}}{n}}\mu(|u_t|)$ can be treated.

\subsection{Difficulties in test function methods}
\hspace{5mm}We stress that our test function $\varphi(t,x)$ introduced in the proofs of Theorems \ref{Thm-Blow-up} and \ref{Thm-Upper-Bound} is not simply a generalization of the one used in \cite{D'Abbicco-Fujiwara=2021} (where some blow-up results for problem \eqref{Sigma-Evolution-Power} were studied in the case $\sigma,\delta\in \mb R_+\backslash \mb N_+$), or in \cite{Ebert-Girardi-Reissig=2020} (where the general nonlinearity $|u|^p\mu(|u|)$ was considered). \\
In particular, the necessity to construct new test functions is motivated by the following main difficulties which arise simultaneously in problem \eqref{eq:CPgeneral}:
\begin{itemize}
	\item if $\psi=\psi(t,x)$ is a classical test function with compact support and $\bar \sigma\in \mb R_+\backslash \mb N_+$, then $(-\Delta)^{\bar{\sigma}}\psi^2 $ in general do not have compact support;  in particular, one cannot carry out the estimate $|(-\Delta)^{\bar{\sigma}}\psi^2|\leqslant 2|\psi\,(-\Delta)^{\bar\sigma} \psi|$  (cf. \cite[Page 3]{D'Abbicco-Fujiwara=2021} for detailed explanations); 
	\item the presence of a modulus of continuity in the nonlinearity $|u|^{p_{\mathrm{c}}(n)}\mu(|u|)$ makes necessary the use of an improved test function method (introduced in \cite{Ikeda-Sobajima=2019}) to get optimal blow-up results and sharp upper bound estimates of the lifespan; among the various novelties, this method employs cut-off functions with cone-type domains. 
\end{itemize}
In Lemma \ref{Lem-Varphi-Fractional-Derivatives} we will provide some useful estimates for our test functions which are essential to prove the desired results.

\section{Main results}\setcounter{equation}{0}\label{Section-Main-Result}

\subsection{On the critical regularity of nonlinearities}
\hspace{5mm}We begin by stating the global in time existence result for the semilinear evolution equation \eqref{Sigma-Evolution-Modulus} under the Dini condition \eqref{Dini-Condition} on $\mu$; in particular, we are mainly interested to infinitesimal functions $\mu(\tau)$ which tend to $0$ as $\tau\to0^+$ slower than any power $\{\tau^{\beta}\}_{\beta\in(0,1]}$ (see later Example \ref{Example-Dini-critical}).
\begin{theorem}\label{Thm-GESDS} Let $\sigma\geqslant 1$ and $\delta\in[0,\sigma]$. Let $\mu:[0,+\infty)\to[0,+\infty)$ be a modulus of continuity fulfilling the following assumptions:
	\begin{description}
		\item[$(A_1)$] $\mu\in\ml{C}^1(\mb{R}_+)$ and there exists $\tau_1>0$ such that $|\mu'(\tau)|\lesssim \tau^{-1}\mu(\tau)$ for any $\tau\in(0,\tau_1)$;
		\item[$(A_2)$] $\mu$ satisfies the Dini condition, i.e. \eqref{Dini-Condition} holds true for some (and then for any) $\tau_0>0$.
	\end{description}
	Moreover, assume $2\delta<n<2\sigma$ in the effective case $\delta\in[0,\sigma/2)$, $\sigma<n$ in the limit case $\delta=\sigma/2$, $n=1$ with $\sigma\in(2/3,1)$ or $1<\sigma<n\leqslant\min\{2\sigma,\bar{n}(\sigma)\}$  in the non-effective case $\delta\in(\sigma/2,\sigma]$ where
	\begin{align*}
		\bar{n}(\sigma):=\frac{3\sigma-2}{2}\left(\sqrt{1+8\sigma(3\sigma-2)^{-2}}+1\right).
	\end{align*}
	Suppose also
	\begin{align*}
		(u_0,u_1)\in\ml{A}_{\delta,\sigma}:=\begin{cases}
			(H^{\sigma}\cap L^1)\times (L^2\cap L^1)&\mbox{if}\ \  \delta\in[0,\sigma/2),\\
			(L^{\infty}\cap L^1)\times(L^{\frac{n}{\sigma}}\cap L^1)&\mbox{if}\ \ \delta=\sigma/2,\\
			(H^{2\delta}_m\cap L^1)\times (L^{m}\cap L^1)&\mbox{if}\ \ \delta\in(\sigma/2,\sigma],
		\end{cases}
	\end{align*}
	for some $m>1$ such that $2\delta m>n$.
	Then, there exists $\varepsilon_0>0$ such that for any $\varepsilon\in(0,\varepsilon_0]$ the semilinear Cauchy problem \eqref{Sigma-Evolution-Modulus} has a uniquely determined global in time  solution
	\begin{align*}
		u\in\ml{C}([0,+\infty),L^{p_{\mathrm{c}}(n)}\cap L^{\infty}),
	\end{align*}
	satisfying the following decay estimate:
	\begin{align}\label{Est-L^q}
		\|u(t,\cdot)\|_{L^q}\leqslant C\varepsilon(1+t)^{-\frac{1}{\kappa}\left(n-\min\{2\delta,\sigma\}-\frac{n}{q}\right)}\|(u_0,u_1)\|_{\ml{A}_{\delta,\sigma}}
	\end{align}
	for any $p_{\mathrm{c}}(n)\leqslant q\leqslant+\infty$, where $C>0$ is independent of $t$ and $\varepsilon$.
\end{theorem}

\begin{remark}
	\label{rem:A1}
	Defining the nonlinearity $g=g(\tau)$ by
	\begin{align}\label{g-function}
		g(\tau):=|\tau|^{p_{\mathrm{c}}(n)}\mu(|\tau|),
	\end{align}
	the assumption $(A_1)$ can be replaced by the condition
	\[ 	|g'(\tau)|\lesssim |\tau|^{p_{\mathrm{c}}(n)-1}\mu(|\tau|) \ \ \mbox{for any}\ \ |\tau|\in(0,\tau_1),\]
	which is a direct consequence of $(A_1)$ and would be sufficient to prove Theorem \ref{Thm-GESDS}.
\end{remark}

Next, we study the optimality of the Dini condition for our model \eqref{Sigma-Evolution-Modulus} by proving blow-up of weak solutions under the non-Dini condition \eqref{not-Dini-Condition}. Before proposing the definition of weak solutions, we firstly introduce a suitable function space for test functions.
\begin{defn}\label{Defn-test-space}
	Let  $q_0$ and $s_0$ defined by $\eqref{eq:s0}$. For any $T>0$ the space $\mathfrak{X}_{q_0}([0,T)\times\mb{R}^n)$ is the set of all functions $\varphi:[0,T)\times\mb{R}^n\to \mb{R}$ such that $\varphi\in \ml{C}^2([0,T), \ml{C}^{2+2[\sigma]}(\mb R^n))$ and the functions
	\begin{align}
		\label{eq:test_space_assumption1}
		\left(1+t^2+|x|^{2\kappa}\right)^{\frac{q_0}{2\kappa}}\partial_t^j\varphi(t,x)&\ \ \mbox{for}\ \ j=0,1,2,\\
		\label{eq:test_space_assumption2}
		\left(1+t^2+|x|^{2\kappa}\right)^{\frac{q_0}{2\kappa}}(-\Delta)^{\bar{\sigma}}\partial_t^j\varphi(t,x)&\ \ \mbox{for}\ \ j=0,1,
	\end{align}
	are uniformly bounded for any $\bar\sigma$ integer in $\{0,\dots, [\sigma] +1\}$; additionally, if $\sigma\notin \mb N_+$ the functions in \eqref{eq:test_space_assumption1} and \eqref{eq:test_space_assumption2} are uniformly bounded for any $0< \bar{\sigma}< [\sigma]+1$ satisfying $\{\bar{\sigma}\}\in[s_0,1)$.
\end{defn}
\begin{remark}
	The vector space $\mathfrak{X}_{q_0}([0,T)\times\mb{R}^n)$ is non-empty; for example,
	\begin{align*}
		\left(1+t^{2\theta}+|x|^{2\theta\kappa}\right)^{-r_0}\in\mathfrak{X}_{q_0}([0,T)\times\mb{R}^n)
	\end{align*}
	if $\theta\geqslant\max\{1,([\sigma]+1)/\kappa \}$ and $r_0\geqslant q_0/(2\theta\kappa)$ (see later Lemma \ref{Lem-Varphi-Fractional-Derivatives}).\\ Moreover, if $\varphi\in\mathfrak{X}_{q_0}([0,T)\times\mb{R}^n)$, we notice $\varphi(t,\cdot)\in L^{\infty}(\mb{R}^n;\langle x\rangle^{q_0}\,\mathrm{d}x)\subset L^1(\mb R^n)$ for any $t\in[0,T)$, due to $q_0>n$, and similarly, $(-\Delta)^{\bar{\sigma}}\varphi(t,\cdot)\in L^1(\mb R^n)$  for any $t\in[0,T)$ and $\bar{\sigma}\geqslant 0$ chosen as in Definition \ref{Defn-test-space}.
\end{remark}
\begin{defn}\label{Defn-Weak-Solution}
	Let $\sigma\geqslant 1$, $\delta\in[0,\sigma]$ and $T\in(0,+\infty]$. Let us consider $u_0,u_1\in L^1(\mb{R}^n;\langle x\rangle^{-q_0}\,\mathrm{d}x)$,  with $q_0$ and $s_0$ given in Definition \ref{Defn-test-space}. We say that $u=u(t,x)$ is a weak solution to the semilinear Cauchy problem \eqref{Sigma-Evolution-Modulus}  if
	\begin{align*}
		u\in L_{\mathrm{loc}}^{p_{\mathrm{c}}(n)}\left([0,T), L^{p_{\mathrm{c}}(n)}(\mb{R}^n;\langle x\rangle^{-q_0}\,\mathrm{d}x)\right),
	\end{align*}
	and $u$ fulfills the integral relation
	\begin{align}
		&\int_0^T\rho(t)\int_{\mb{R}^n}g\big(u(t,x)\big)\varphi(t,x)\,\mathrm{d}x\,\mathrm{d}t\notag\\
		&=\int_0^T\int_{\mb{R}^n}u(t,x)\left(\partial_t^2+(-\Delta)^{\sigma}-(-\Delta)^{\delta}\partial_t\right)\big(\rho(t)\,\varphi(t,x)\big)\,\mathrm{d}x\,\mathrm{d}t-\varepsilon\int_{\mb{R}^n}u_1(x)\,\rho(0)\,\varphi(0,x)\,\mathrm{d}x\notag\\
		&\quad+\varepsilon\int_{\mb{R}^n}u_0(x)\left(\rho'(0)\,\varphi(0,x)+\rho(0)\,\varphi_t(0,x)-\rho(0)\,(-\Delta)^{\delta}\varphi(0,x)\right)\mathrm{d}x\label{Equality-Weak-Solution}
	\end{align}
	for any $\varphi\in\mathfrak{X}_{q_0}([0,T)\times\mb{R}^n)$ and $\rho\in\ml{C}^2_0([0,T))$.
\end{defn}

\begin{remark}
	All the integral terms in Definition \ref{Defn-Weak-Solution} are well-defined. Indeed, since $\mathrm{supp}\,\rho\subset[0,T_0]$ with $T_0<T$, we can employ the decay properties of $\varphi\in\mathfrak{X}_{q_0}([0,T_0)\times\mb{R}^n)$ described in Definition \ref{Defn-test-space} to get
	\begin{align*}
		&\int_0^{T}\int_{\mb{R}^n}|u(t,x)|\left|\left(\partial_t^2+(-\Delta)^{\sigma}-(-\Delta)^{\delta}\partial_t\right)\big(\rho(t)\,\varphi(t,x)\big)\right|\,\mathrm{d}x\,\mathrm{d}t\\
		&\lesssim\left(\int_0^{T_0}\int_{\mb{R}^n}|u(t,x)|^{p_{\mathrm{c}}(n)}\langle x\rangle^{-q_0}\,\mathrm{d}x\,\mathrm{d}t\right)^{\frac{1}{p_{\mathrm{c}}(n)}}\left(\int_0^{T_0}\int_{\mb{R}^n}\langle x\rangle^{\frac{q_0p'_{\mathrm{c}}(n)}{p_{\mathrm{c}}(n)}}\left(1+t^2+|x|^{2\kappa}\right)^{-\frac{q_0p'_{\mathrm{c}}(n)}{2\kappa}}\mathrm{d}x\,\mathrm{d}t\right)^{\frac{1}{p'_{\mathrm{c}}(n)}}\\
		&\lesssim \|u\|_{L_{\mathrm{loc}}^{p_{\mathrm{c}}(n)}\left([0,T_0), L^{p_{\mathrm{c}}(n)}(\mb{R}^n;\langle x\rangle^{-q_0}\,\mathrm{d}x)\right)}\left(\int_0^{T_0}\int_{\mb{R}^n}\langle x\rangle^{-q_0}\,\mathrm{d}x\,\mathrm{d}t\right)^{\frac{1}{p'_{\mathrm{c}}(n)}}<+\infty,
	\end{align*}
	being $q_0>n$. Analogously, since $u_0,u_1\in L^1(\mb{R}^n;\langle x\rangle^{-q_0}\,\mathrm{d}x)$ and
	\begin{align*}
		|\varphi(0,x)|+|\varphi_t(0,x)|+|(-\Delta)^{\delta}\varphi(0,x)|\lesssim\langle x\rangle^{-q_0},
	\end{align*}
	we can conclude that the integrals for $u_0,u_1$ are finite.
	
\end{remark}



In the following we state our blow-up result which holds true even in the case of fractional Laplacians.
\begin{theorem}\label{Thm-Blow-up}
	Let $\sigma\geqslant 1$ and $\delta\in[0,\sigma]$. Let $\mu:[0,+\infty)\to[0,+\infty)$ be a modulus of continuity fulfilling the following assumptions:
	\begin{description}
		\item[$(A_3)$] the function $g:\tau\in\mb{R}\to |\tau|^{p_{\mathrm{c}}(n)}\mu(|\tau|)\in[0,+\infty)$ is convex;
		\item[$(A_4)$] $\mu$ satisfies the non-Dini condition, i.e. \eqref{not-Dini-Condition} holds for some (and, then, for any) $\tau_0>0$.
	\end{description}
	We also assume that the Cauchy data $u_0$ and $u_1$ satisfy
	\begin{align}
		\label{eq:data_sign_delta=0}u_0,\, u_1\in L^1(\mb R^n; \langle x\rangle^{-q_0}\,\mathrm{d}x) \quad \text{and} \quad \int_{\mb R^n} \big(u_0(x)+u_1(x)\big)\langle x\rangle^{-q_0}\,\mathrm{d}x >0,
	\end{align}
	if $\delta=0$; whereas if $\delta>0$ we suppose that $u_1$ belongs to $L^1(\mb R^n; \langle x\rangle^{-q_0}\,\mathrm{d}x)$, and additionally one of the following conditions holds:
	\begin{subequations}
		\begin{align}
			\label{eq:data_sign_delta>0_1}  u_0\in L^1(\mb R^n) \quad &\text{and} \quad  \int_{\mb R^n}u_1(x)\langle x\rangle^{-q_0}\,\mathrm{d}x >0,\\
			\label{eq:data_sign_delta>0_2} u_0\in L^1(\mb R^n; \langle x\rangle^{-q_0}\,\mathrm{d}x) \quad &\text{and} \quad \exists \bar{c}>0 : u_1(x)\geqslant 2\bar{c} |u_0(x)| \quad \forall x\in \mb R^n.
		\end{align}
	\end{subequations}
	Thus, if $u$ is a  local in time weak solution to the semilinear Cauchy problem \eqref{Sigma-Evolution-Modulus} on $[0,T_{\varepsilon})$ according to Definition \ref{Defn-Weak-Solution}, then $u$ blows up in finite time, i.e. $T_{\varepsilon}<+\infty$.
\end{theorem}

Combining the results from Theorems \ref{Thm-GESDS} and \ref{Thm-Blow-up}, we see that the Dini condition \eqref{Dini-Condition} for the modulus of continuity $\mu$ in the semilinear evolution equation \eqref{Sigma-Evolution-Modulus} with $\delta\in[0,\sigma]$,  is a threshold condition that separates the global in time existence of small data solutions from the blow-up of local in time solutions. These results generalize those in \cite{D'Abbicco-Girardi=2023,Girardi=2024}, where only  the case with $\sigma$ and $\delta$ in $\mb N_0$ was considered.

\begin{exam}\label{Example-Dini-critical}
	Let $\tau_0>0$ be sufficiently small. The assumptions $(A_3)-(A_4)$ in Theorem \ref{Thm-Blow-up} hold true if $\mu: [0,+\infty)\to [0,+\infty)$ is a modulus of continuity defined by the following functions in a small interval $(0,\tau_0]$:
	\begin{description}
		\item[(1MoC)] $\mu(\tau)=(\ln\frac{1}{\tau})^{-\gamma}$ with $\gamma\in(0,1]$;
		\item[(2MoC)] $\displaystyle{\mu(\tau)=\prod_{j=1}^{k-1}\left(\ln^{[j]}\tfrac{1}{\tau}\right)^{-1}\left(\ln^{[k]}\tfrac{1}{\tau}\right)^{-\gamma}  }$ with $\gamma\in(0,1]$ and $k\in \mathbb{N}_+$, $k\geqslant 2$.
	\end{description}
	Therefore, for these choices of $\mu$ the blow-up result holds true, provided that the initial data fulfill the sign conditions required in Theorem \ref{Thm-Blow-up}.\\
	On the other hand, when $\gamma\in(1,+\infty)$ in (1MoC) and (2MoC), these functions $\mu(\tau)$ satisfy the assumptions $(A_1)-(A_2)$ from Theorem \ref{Thm-GESDS}, i.e. the global in time existence result is valid. Namely, in these examples the power $\gamma=1$ determines a critical threshold in the existence of global in time weak solutions to \eqref{Sigma-Evolution-Modulus}.
\end{exam}

\subsection{On the sharp estimates of the lifespan}\label{Subsection-Sharp-lifespan}
\hspace{5mm}As we explained in Theorem \ref{Thm-Blow-up}, under the non-Dini condition \eqref{not-Dini-Condition} the non-trivial local in time solution may blow up in finite time; this motivates us to provide more detailed information about the lifespan $T_{\varepsilon}$. In the following we will state two new results which explain the upper and lower bounds estimates of the lifespan, respectively; as a result, the sharp estimates \eqref{Sharp-Lifespan-Est} of the lifespan for the Cauchy problem \eqref{Sigma-Evolution-Modulus} will follow. 

\begin{theorem}\label{Thm-Lower-Bound}
	Let $\sigma\geqslant 1$ and $\delta\in[0,\sigma]$; suppose that $\mu:[0,+\infty)\to[0,+\infty)$ is a modulus of continuity fulfilling the assumptions $(A_1)$ and $(A_4)$. Then, there exist $\varepsilon_0>0$ and three constants $k_1,k_2,K>0$, independent of $\varepsilon$, such that for any $\varepsilon\in(0,\varepsilon_0]$ and any $T>0$ satisfying 
	\begin{align}\label{Lower-Bound-Est-Lifespan}
		T\leqslant T_{\varepsilon,\ell}:= K\varepsilon^{\frac{\kappa}{n-\min\{2\delta,\sigma\}}}\left[\ml{H}^{-1}\left(k_1\varepsilon^{-\frac{2\sigma}{n-\min\{2\delta,\sigma\}}}+\ml{H}(k_2\varepsilon)\right)\right]^{-\frac{\kappa}{n-\min\{2\delta,\sigma\}}},
	\end{align}
	the semilinear Cauchy problem \eqref{Sigma-Evolution-Modulus} has a uniquely determined local in time solution
	\begin{align*}
		u\in\ml{C}([0,T],L^{p_{\mathrm{c}}(n)}\cap L^{\infty})
	\end{align*}
	satisfying the estimate \eqref{Est-L^q} for any $t\in[0,T]$.
\end{theorem}

\begin{theorem}\label{Thm-Upper-Bound}Let $\sigma\geqslant 1$ and $\delta\in [0,\sigma]$; suppose that  $\mu:[0,+\infty)\to[0,+\infty)$ is a modulus of continuity fulfilling the assumptions $(A_3)$ and $(A_4)$; assume also that the Cauchy data $u_0$ and $u_1$  satisfy the same assumptions of Theorem \ref{Thm-Blow-up}. Thus,  there exists $\bar\varepsilon>0$ (depending only on $\ml H$, $u_0$ and $u_1$) such that, for any $\varepsilon\in (0, \bar\varepsilon)$, if there exists a weak solution $u$ to \eqref{Sigma-Evolution-Modulus} in $ L_{\mathrm{loc}}^{p_{\mathrm{c}}(n)}\left([0,T), L^{p_{\mathrm{c}}(n)}(\mb{R}^n;\langle x\rangle^{-q_0}\,\mathrm{d}x)\right)$ according to Definition \ref{Defn-Weak-Solution}, then $T$ satisfies the following upper bound estimate:
	\begin{align}\label{Upper-Bound-Lifespan}
		T\leqslant T_{\varepsilon,u}:= \widetilde{K}\varepsilon^{\frac{\kappa}{n-\min\{2\delta,\sigma\}}}\left[\ml{H}^{-1}\left(\widetilde{k}_1\varepsilon^{-\frac{2\sigma}{n-\min\{2\delta,\sigma\}}}+\ml{H}(\widetilde{k}_2\varepsilon)\right)\right]^{-\frac{\kappa}{n-\min\{2\delta,\sigma\}}},
	\end{align}
	for some constants $\widetilde{k}_1,\widetilde{k}_2,\widetilde{K}>0$, independent of $\varepsilon$.
\end{theorem}

Summarizing, if the assumptions of Theorem \ref{Thm-Lower-Bound} and Theorem \ref{Thm-Upper-Bound} are satisfied, then problem \eqref{Sigma-Evolution-Modulus} admits a local solution defined on  $[0,T_{\varepsilon})$ if, and only if, $T_{\varepsilon}$ satisfies $T_{\varepsilon,\ell}\leqslant T_{\varepsilon} \leqslant T_{\varepsilon,u}$; as a consequence, one  concludes the sharp estimates \eqref{Sharp-Lifespan-Est} for the lifespan $T_{\varepsilon}$. 

\subsection{Some examples for the sharp lifespan estimates}\label{Sub-section-Examples-lifespan}
\hspace{5mm}This section provides some examples of sharp lifespan estimates \eqref{Sharp-Lifespan-Est} when $\mu$ is a modulus of continuity such that $\ml{H}(0)=+\infty$.

We begin by deducing a rougher estimate than the one in \eqref{Sharp-Lifespan-Est}. Later, we will use it in some cases instead of \eqref{Sharp-Lifespan-Est} to better emphasize the dominant $\varepsilon$-dependent term in the sharp estimates of $T_{\varepsilon}$. 
Recalling $\tau_0<1$, we notice that 
\begin{align*}
	\lim\limits_{\varepsilon\to 0^+}\frac{\ml{H}(c\,\varepsilon)}{C\varepsilon^{-\frac{2\sigma}{n-\min\{2\delta,\sigma\}}}}&=C^{-1} \lim\limits_{\varepsilon\to 0^+} \varepsilon^{\frac{2\sigma}{n-\min\{2\delta,\sigma\}}}\int_{c\,\varepsilon}^{\tau_0}\frac{\mu(\varrho)}{\varrho}\,\mathrm{d}\varrho\\
	& \leqslant  C^{-1}\max_{\tau\in [0,\tau_0]}\{\mu(\tau)\} \lim\limits_{\varepsilon\to 0^+} \big( \varepsilon^{\frac{2\sigma}{n-\min\{2\delta,\sigma\}}}(-\ln(c\,\varepsilon))\big)=0.
\end{align*}
Therefore, in \eqref{Sharp-Lifespan-Est} the term $\varepsilon^{-\frac{2\sigma}{n-\min\{2\delta,\sigma\}}}$ is dominant in comparison with $\ml{H}(c\,\varepsilon)$ for $0<\varepsilon\ll 1$, which allows us to reduce it into
\begin{align}\label{Lifespan-Approximation}
	T_{\varepsilon}\approx\varepsilon^{\frac{\kappa}{n-\min\{2\delta,\sigma\}}}\left[\ml{H}^{-1}\left(2C\varepsilon^{-\frac{2\sigma}{n-\min\{2\delta,\sigma\}}}\right)\right]^{-\frac{\kappa}{n-\min\{2\delta,\sigma\}}}
\end{align}
for $\varepsilon\in(0,\varepsilon_0]$ with a sufficiently small constant $\varepsilon_0>0$.

We consider two typical modulus of continuity (weaker than any H\"older's modulus of continuity) in Example \ref{Example-Dini-critical}, i.e. (1MoC) and (2MoC), whose corresponding auxiliary functions $\ml{H}(\tau)$ and $\ml{H}^{-1}(\omega)$ have been explicitly evaluated in \cite[Section 5]{Chen-Palmieri=2024}. Here, we omit the details.
\begin{exam}
	Let $\mu(\tau)=(\ln\frac{1}{\tau})^{-\gamma}$ for any $\tau\in (0,\tau_0]$ with $\gamma\in(0,1]$; by employing \eqref{Lifespan-Approximation} for $\gamma\in (0,1)$ and \eqref{Sharp-Lifespan-Est} for $\gamma=1$ we obtain
	\begin{align*}
		T_{\varepsilon}\approx
		\begin{cases}
			\displaystyle{ \varepsilon^{\frac{\kappa}{n-\min\{2\delta,\sigma\}}}\exp\left(C\varepsilon^{-\frac{2\sigma}{(n-\min\{2\delta,\sigma\})(1-\gamma)}}\right)}&\mbox{if}\ \ \gamma\in(0,1),\\[0.8em]
			\displaystyle{ \varepsilon^{\frac{\kappa}{n-\min\{2\delta,\sigma\}}}\exp\left[c\left(\ln\frac{1}{\varepsilon}\right)\exp\left( C\varepsilon^{-\frac{2\sigma}{n-\min\{2\delta,\sigma\}}}\right)\right]}&\mbox{if}\ \ \gamma=1,
		\end{cases}
	\end{align*}
	where the positive constants $c$ and $C$ depend on $\tau_0,n,\sigma,\delta,\gamma$ only. We underline that as $\gamma\to0^+$ the value of $T_\varepsilon$ tends to the  sharp lifespan estimate for the semilinear evolution equations \eqref{Sigma-Evolution-Power} with the critical power nonlinearity $|u|^{p_{\mathrm{c}}(n)}$, explicitly evaluated in Proposition \ref{Coro-Classical-sigma-evolution}.
\end{exam}

\begin{exam}
	Let us suppose
	\begin{equation}
		\label{eq:log-iterations}
		\mu(\tau)=\bigg(\ln^{[k]}\frac{1}{\tau}\bigg)^{-\gamma}\,\prod\limits_{j=1}^{k-1}\bigg(\ln^{[j]}\frac{1}{\tau}\bigg)^{-1} \quad \text{for any} \quad \tau\in (0,\tau_0], 
	\end{equation}
	for some $k\geqslant 2$ and  $\gamma\in(0,1]$; then employing \eqref{Lifespan-Approximation} for $\gamma\in (0,1)$ and \eqref{Sharp-Lifespan-Est} for $\gamma=1$ we obtain
	\begin{align*}
		T_{\varepsilon}\approx
		\begin{cases}
			\displaystyle{ \varepsilon^{\frac{\kappa}{n-\min\{2\delta,\sigma\}}}\exp\left[c\exp^{[k-1]}\left( C\varepsilon^{-\frac{2\sigma}{(n-\min\{2\delta,\sigma\})(1-\gamma)}}\right)\right]}&\mbox{if}\ \ \gamma\in(0,1),\\[0.8em]
			\displaystyle{  \varepsilon^{\frac{\kappa}{n-\min\{2\delta,\sigma\}}}\exp\left\{c \exp^{[k-1]}\left[\left(\ln^{[k]}\tfrac{1}{\varepsilon}\right)\exp\left(C\varepsilon^{-\frac{2\sigma}{n-\min\{2\delta,\sigma\}}}\right)\right]\right\}}&\mbox{if}\ \ \gamma=1,
		\end{cases}
	\end{align*}
	where the positive constants $c$ and $C$ depend on $\tau_0,n,\sigma,\delta,\gamma$ only. It shows that the lifespan becomes longer if one considers a larger number $k$ of iterations in the definition \eqref{eq:log-iterations} of $\mu$.
\end{exam}

\section{Existence results and sharp lower bound estimates of the lifespan}\setcounter{equation}{0}\label{Section-Existence}
\subsection{Philosophy of our proofs}\label{Subsection-Philosophy}
\hspace{5mm}Before proving the local/global in time existence results and the sharp lower bound estimates of the lifespan $T_{\varepsilon}$, we address some preliminary results and explain our approach, whose philosophy also can be used in deriving sharp lower bound estimates of the lifespan for other related semilinear evolution models.

For $T>0$ we introduce the evolution space $X_T$ of solutions defined by $X_T:=\ml{C}([0,T],L^{p_{\mathrm{c}}(n)}\cap L^{\infty})$, with the critical exponent $p_{\mathrm{c}}(n)$ defined in \eqref{Critical-Exponent}, equipped with the time-weighted norm
\begin{align*}
	\|u\|_{X_T}:=\sup\limits_{t\in[0,T]}\left((1+t)^{\frac{1}{p_{\mathrm{c}}(n)}}\|u(t,\cdot)\|_{L^{p_{\mathrm{c}}(n)}}+(1+t)^{\frac{n-\min\{2\delta,\sigma\}}{\kappa}}\|u(t,\cdot)\|_{L^{\infty}}\right).
\end{align*}
Here, the time-dependent weights are strongly motivated by some known decay estimates for the solution to the corresponding linear Cauchy problem which will be recalled in Subsection \ref{Sub-section-detail} (see \cite{D'Abbicco-Ebert=2017,D'Abbicco-Ebert=2021,D'Abbicco-Ebert=2022, Pham-K-Reissig=2015} and references therein, according to the different values of $\delta$ and $\sigma$ ).

Let us consider the following nonlinear integral operator:
\begin{align*}
	\ml{N}:\ u(t,x)\in X_T\to \ml{N}[u](t,x):=\varepsilon u_{\lin}(t,x)+u_{\nlin}(t,x)
\end{align*}
for any $t\in[0,T]$ and $x\in\mb{R}^n$, where we define $u_{\lin}=u_{\lin}(t,x)$ by
\begin{align*}
	u_{\lin}(t,x):=\underbrace{\frac{\lambda_1(|D|)\,\mathrm{e}^{\lambda_2(|D|)t}-\lambda_2(|D|)\,\mathrm{e}^{\lambda_1(|D|)t}}{\lambda_1(|D|)-\lambda_2(|D|)}}_{=:K_0(t,|D|)}u_0(x)+\underbrace{\frac{\mathrm{e}^{\lambda_1(|D|)t}-\mathrm{e}^{\lambda_2(|D|)t}}{\lambda_1(|D|)-\lambda_2(|D|)}}_{=:K_1(t,|D|)}u_1(x)
\end{align*}
with $\lambda_{1,2}(|D|):=(-|D|^{2\delta}\pm\sqrt{|D|^{4\delta}-4|D|^{2\sigma}}\,)/2$, and $u_{\nlin}=u_{\nlin}(t,x)$ by
\begin{align*}
	u_{\nlin}(t,x):=\int_0^tK_1(t-\eta,|D|)\,g\big(u(\eta,x)\big)\,\mathrm{d}\eta
\end{align*}
with the nonlinearity $g$ introduced in \eqref{g-function}. By applying Duhamel's principle, we know that $u\in X_T$ is a solution to the semilinear Cauchy problem \eqref{Sigma-Evolution-Modulus} if, and only if, it is a fixed point for the operator $\ml N$.

The assumption $(u_0,u_1)\in\ml{A}_{\delta,\sigma}$ ensures that the solution to the linear problem associated to \eqref{Sigma-Evolution-Modulus} satisfies suitable long-time decay estimates, recalled in Subsection \ref{Sub-section-detail}, which allows to conclude $u_{\lin}\in X_T$ for any $T>0$ and to obtain the uniform estimate 
\begin{align}\label{linear-Data-A}
	\|u_{\lin}\|_{X_T}\leqslant C_0 \|(u_0,u_1)\|_{\ml{A}_{\delta,\sigma}},
\end{align}
for a constant $C_0>0$.

Our goal is to prove the existence of a unique fixed point  $u$ of the nonlinear integral operator $\ml{N}$ in the space $X_T$, which coincides with the unique solution to \eqref{Sigma-Evolution-Modulus} in $X_T$. To achieve our aim we will apply the Banach contraction principle, after proving that for any $u$ and $\tilde{u}$ in the set
\begin{align*}
	\mathfrak{B}_{\nu}(X_{T}):=\{w\in X_{T}:\|w\|_{X_T}\leqslant \nu\}\ \ \mbox{with}\ \ \nu:=2C_0\|(u_0,u_1)\|_{\ml{A}_{\delta,\sigma}}\varepsilon>0,
\end{align*} 
the following fundamental inequalities hold uniformly:
\begin{align}
	\|\ml{N}[u]\|_{X_T}&\leqslant  C_0\|(u_0,u_1)\|_{\ml{A}_{\delta,\sigma}}\varepsilon+C_1(\mu,\nu,T)\|u\|_{X_T}^{p_{\mathrm{c}}(n)},\label{Crucial-01}\\
	\|\ml{N}[u]-\ml{N}[\tilde{u}]\|_{X_T}&\leqslant C_1(\mu,\nu,T)\|u-\tilde{u}\|_{X_T}\left(\|u\|_{X_T}^{p_{\mathrm{c}}(n)-1}+\|\tilde{u}\|_{X_T}^{p_{\mathrm{c}}(n)-1}\right),\label{Crucial-02}
\end{align} 
for a suitable constant $C_1(\mu,\nu,T)>0$ depending only on $\nu$, $T$ and the modulus of continuity $\mu$. 
For simplicity, we denote
\begin{align*}
	\ml{W}_{[u,\tilde{u}],X_T}^{p_{\mathrm{c}}(n)}:=\|u-\tilde{u}\|_{X_T}\left(\|u\|_{X_T}^{p_{\mathrm{c}}(n)-1}+\|\tilde{u}\|_{X_T}^{p_{\mathrm{c}}(n)-1}\right).
\end{align*}

Notice that from the combination of \eqref{linear-Data-A} and \eqref{Crucial-02} with $\tilde{u}=0$ the desired inequality \eqref{Crucial-01} directly follows.

For the sake of readability, we state the philosophy of our proofs. In the forthcoming subsections we will demonstrate the crucial estimate
\begin{align}\label{Crucial-04}
	C_1(\mu,\nu,T)&\lesssim\int_0^{T}(1+\eta)^{-1}\mu\left(C_2(1+\eta)^{-\frac{n-\min\{2\delta,\sigma\}}{\kappa}}\nu\right)\mathrm{d}\eta\notag\\
	&\leqslant C_3 \int^{C_2\nu}_{C_2(1+T)^{-\frac{n-\min\{2\delta,\sigma\}}{\kappa}}\nu}\frac{\mu(\tau)}{\tau}\,\mathrm{d}\tau
\end{align}
with suitable constants $C_2,C_3>0$ independent of $T$ and $\nu$. We next separate our discussion according to the integrability of $\mu(\tau)/\tau$ over $[0,\tau_0]$. Here, we take $\tau_0=C_2\nu$.
\begin{description}
	\item[(a) The Dini condition holds.] By using the Dini condition \eqref{Dini-Condition} we can pick $\varepsilon$ (and then $\nu$ being sufficiently small) such that
	\begin{align*}
		C_1(\mu,\nu,T)\leqslant C_3 \int^{C_2\nu}_{0}\frac{\mu(\tau)}{\tau}\,\mathrm{d}\tau\leqslant \frac{1}{4}\nu^{1-p_{\mathrm{c}}(n)};
	\end{align*}
	indeed, taking $\nu\downarrow 0$ the integral term becomes arbitrarily small, instead the power $\nu^{1-p_{\mathrm{c}}(n)}$ tends to $+\infty$ being $p_{\mathrm{c}}(n)>1$.\\
	Therefore, the coefficient $C_1(\mu,\nu,T)$ estimated in \eqref{Crucial-04} is uniformly bounded with respect to $T>0$.
	Then, by taking $T=+\infty$, \eqref{Crucial-01} and \eqref{Crucial-02} imply, respectively,
	\begin{align}\label{star-01}
		\|\ml{N}[u]\|_{X_{+\infty}}\leqslant \frac{3}{4}\nu\ \ \mbox{and}\ \ \|\ml{N}[u]-\ml{N}[\tilde{u}]\|_{X_{+\infty}}\leqslant \frac{1}{2}\|u-\tilde{u}\|_{X_{+\infty}}
	\end{align}
	for any $u,\tilde{u}\in\mathfrak{B}_{\nu}(X_{+\infty})$. As a consequence, applying the Banach fixed point theorem, it follows the existence of a unique  global in time solution $u\in\mathfrak{B}_{\nu}(X_{+\infty})$. As a byproduct, we also find
	\begin{align*}
		\sup\limits_{t\in[0,+\infty]}\left(\sum\limits_{q=p_{\mathrm{c}}(n),+\infty}(1+t)^{\frac{1}{\kappa}\left(n-\min\{2\delta,\sigma\}-\frac{n}{q}\right)}\|u(t,\cdot)\|_{L^q}\right)\lesssim\varepsilon\|(u_0,u_1)\|_{\ml{A}_{\delta,\sigma}}.
	\end{align*}
	The interpolation between $L^{p_{\mathrm{c}}(n)}$ and $L^{\infty}$ shows the desired estimate \eqref{Est-L^q}.
	\item[(b) The non-Dini condition holds.] Recalling the definition of the function $\ml{H}$ in \eqref{H-function}, one may rewrite \eqref{Crucial-04} as
	\begin{align*}
		C_1(\mu,\nu,T)\leqslant C_3\left[\ml{H}\left(C_2(1+T)^{-\frac{n-\min\{2\delta,\sigma\}}{\kappa}}\nu \right)-\ml{H}(C_2\nu)\right].
	\end{align*}
	Let us introduce the following three positive constants independent of $\varepsilon$:
	\begin{align*}
		k_1:=(4C_3)^{-1}\left(2C_0\|(u_0,u_1)\|_{\ml{A}_{\delta,\sigma}}\right)^{1-p_{\mathrm{c}}(n)}, \ \ k_2:=2C_0C_2\|(u_0,u_1)\|_{\ml{A}_{\delta,\sigma}},\ \ K:=k_2^{\frac{\kappa}{n-\min\{2\delta,\sigma\}}}.
	\end{align*}
	Plugging the previous inequality into \eqref{Crucial-01} and \eqref{Crucial-02}, similarly to \eqref{star-01} we find that $\ml{N}$ is a contraction on the ball $\mathfrak{B}_{\nu}(X_T)$ for any $T>0$ satisfying
	\begin{align}\label{LLLL-01}
		&\ml{H}\left(C_2(1+T)^{-\frac{n-\min\{2\delta,\sigma\}}{\kappa}}\nu \right)-\ml{H}(C_2\nu)\leqslant (4C_3)^{-1}\nu^{1-p_{\mathrm{c}}(n)}\notag \\
		&\qquad\Rightarrow \ml{H}\left(k_2\varepsilon(1+T)^{-\frac{n-\min\{2\delta,\sigma\}}{2\kappa}}\right)\leqslant k_1\varepsilon^{1-p_{\mathrm{c}}(n)}+\ml{H}(k_2\varepsilon)\notag\\
		&\qquad\Rightarrow k_2\varepsilon(1+T)^{-\frac{n-\min\{2\delta,\sigma\}}{\kappa}}\geqslant \ml{H}^{-1}\left(k_1\varepsilon^{-\frac{2\sigma}{n-\min\{2\delta,\sigma\}}}+\ml{H}(k_2\varepsilon)\right)\notag\\
		&\qquad\Rightarrow T\leqslant K\varepsilon^{\frac{\kappa}{n-\min\{2\delta,\sigma\}}}\left[\ml{H}^{-1}\left(k_1\varepsilon^{-\frac{2\sigma}{n-\min\{2\delta,\sigma\}}}+\ml{H}(k_2\varepsilon)\right)\right]^{-\frac{\kappa}{n-\min\{2\delta,\sigma\}}}.
	\end{align}
	Thus, for any $T>0$ satisfying \eqref{LLLL-01} the application of Banach fixed point argument allows to obtain the existence of a unique local in time solution $u\in\mathfrak{B}_{\nu}(X_{T})$; hence, the lower bound estimate \eqref{Lower-Bound-Est-Lifespan} for the lifespan $T_{\varepsilon}$ has been proved.
\end{description}

\noindent In conclusion, in order to complete the proof of the global in time existence result in Theorem \ref{Thm-GESDS} and the sharp lower bound estimate of the lifespan  in Theorem \ref{Thm-Lower-Bound}, it remains to demonstrate the crucial estimate \eqref{Crucial-04} under suitable conditions on $\delta,\sigma,n$.

In order to estimate the $L^r$ norm of the nonlinearity $g$ we will employ the following lemma; in its proof the assumption $(A_1)$ plays a crucial role (see also Remark \ref{rem:A1}) .
\begin{lemma}\label{Lemma-Nonlinearity} Assume that the nonlinearity $g=g(\tau)$ defined by \eqref{g-function} satisfies $g\in\ml{C}^1(\mb{R}_+)$ and there exists $\tau_1>0$ such that
	\begin{align}\label{Assumption-Nonlinearity}
		|g'(\tau)|\lesssim |\tau|^{p_{\mathrm{c}}(n)-1}\mu(|\tau|) \ \ \mbox{for any}\ \ |\tau|\in(0,\tau_1).
	\end{align}
	Then, if $|u|+|\tilde{u}|<\tau_1$ the inequality
	\begin{align*}
		\left\|g\big(u(\eta,\cdot)\big)-g\big(\tilde{u}(\eta,\cdot)\big)\right\|_{L^r}&\lesssim \mu\big(\|u(\eta,\cdot)\|_{L^{\infty}}+\|\tilde{u}(\eta,\cdot)\|_{L^{\infty}}\big)\|u(\eta,\cdot)-\tilde{u}(\eta,\cdot)\|_{L^{rp_{\mathrm{c}}(n)}}\\
		&\quad\times\left(\|u(\eta,\cdot)\|_{L^{rp_{\mathrm{c}}(n)}}^{p_{\mathrm{c}}(n)-1}+\|\tilde{u}(\eta,\cdot)\|_{L^{rp_{\mathrm{c}}(n)}}^{p_{\mathrm{c}}(n)-1}\right)
	\end{align*}
	holds true for any $r\geqslant 1$.
\end{lemma}
\begin{proof}
	We can write
	\begin{align*}
		g\big(u(\eta,x)\big)-g\big(\tilde{u}(\eta,x)\big)&=\int_0^1\frac{\partial}{\partial\omega}\left[g\big(\omega u(\eta,x)+(1-\omega)\tilde{u}(\eta,x)\big)\right]\mathrm{d}\omega\\
		&=\int_0^1g'\big(\omega u(\eta,x)+(1-\omega)\tilde{u}(\eta,x)\big)\,\mathrm{d}\omega\,\big(u(\eta,x)-\tilde{u}(\eta,x)\big).
	\end{align*}
	Then, by the  increasing monotonicity of $\mu$ and the assumption \eqref{Assumption-Nonlinearity}, we obtain
	\begin{align*}
		&\big|g\big(u(\eta,x)\big)-g\big(\tilde{u}(\eta,x)\big) \big|\lesssim \int_0^1\left| g'\big( \omega u(\eta,x)+(1-\omega)\tilde{u}(\eta,x)\big)\right| \mathrm{d}\omega\,|u(\eta,x)-\tilde{u}(\eta,x)|\\
		&\qquad\lesssim\mu\big(|u(\eta,x)|+|\tilde{u}(\eta,x)|\big)\,|u(\eta,x)-\tilde{u}(\eta,x)| \left(|u(\eta,x)|^{p_{\mathrm{c}}(n)-1}+|\tilde{u}(\eta,x)|^{p_{\mathrm{c}}(n)-1}\right),
	\end{align*}
	and
	\begin{align}
		\label{eq:mu_Linfty}
		\big\|\mu\big(|u(\eta,\cdot)|+|\tilde{u}(\eta,\cdot)|\big)\big\|_{L^{\infty}}\leqslant\mu\big(\|u(\eta,\cdot)\|_{L^{\infty}}+\|\tilde{u}(\eta,\cdot)\|_{L^{\infty}}\big).
	\end{align}
	Thus, applying H\"older's inequality we obtain
	\begin{align*}
		\left\|g\big(u(\eta,\cdot)\big)-g\big(\tilde{u}(\eta,\cdot)\big)\right\|_{L^r}
		&\lesssim \big\|\mu\big(|u(\eta,\cdot)|+|\tilde{u}(\eta,\cdot)|\big)\big\|_{L^{\infty}} \left\|u(\eta,\cdot)-\tilde{u}(\eta,\cdot)\right\|_{L^{rp_{\mathrm{c}}(n)}}\\
		&\quad\times\left(\|\,|u(\eta,\cdot)|^{p_{\mathrm{c}}(n)-1}\|_{L^{\frac{rp_{\mathrm{c}}(n)}{p_{\mathrm{c}}(n)-1}}}+\|\,|\tilde{u}(\eta,\cdot)|^{p_{\mathrm{c}}(n)-1}\|_{L^{\frac{rp_{\mathrm{c}}(n)}{p_{\mathrm{c}}(n)-1}}}\right)
	\end{align*}
	for any $r\geqslant 1$, which immediately implies our desired estimate thanks to the inequality \eqref{eq:mu_Linfty}.
\end{proof}

\subsection{Proofs of Theorems \ref{Thm-GESDS} and  \ref{Thm-Lower-Bound}}\label{Sub-section-detail}
\hspace{5mm}In what follows we will prove the crucial estimate \eqref{Crucial-04}, under suitable restrictions on the dimension $n$, discussing separately the effective case $\delta\in[0,\sigma/2)$, the limit case $\delta=\sigma/2$, and the non-effective case $\delta\in(\sigma/2,\sigma]$; as a consequence, we will obtain the proof of Theorems \ref{Thm-GESDS} and \ref{Thm-Lower-Bound}. The proof of \eqref{Crucial-04} deeply relies on the application of long-time decay estimates for the solution to the following linear problem associated to \eqref{Sigma-Evolution-Modulus}:
\begin{align}\label{eq:CPlinear}
	\begin{cases}
		u_{tt}+(-\Delta)^{\sigma}u+(-\Delta)^{\delta}u_t=0,&x\in\mb{R}^n,\ t>0,\\
		u(0,x)=\varepsilon u_0(x),\ u_t(0,x)=\varepsilon u_1(x),&x\in\mb{R}^n,
	\end{cases}
\end{align}
already investigated in \cite{D'Abbicco-Ebert=2017,D'Abbicco-Ebert=2021,D'Abbicco-Ebert=2022, Pham-K-Reissig=2015} and some references therein, according to the different values of $\delta$ and $\sigma$.
\paragraph{\large The Effective Case $\delta\in[0,\sigma/2)$}
The estimates of the $L^{p_{\mathrm{c}}(n)}$ and $L^{\infty}$ norms of the solution $u_{\text{lin}}$ to \eqref{eq:CPlinear}, localized at small frequencies, can be found in \cite[Proposition 4.1]{D'Abbicco-Ebert=2017}. Moreover, we can derive an estimate of the  $L^{p_{\mathrm{c}}(n)}$ norm of the solution to \eqref{eq:CPlinear}, localized at large frequencies, employing \cite[Proposition 4.2]{D'Abbicco-Ebert=2017} via the fractional Gagliardo-Nirenberg inequality (thanks to $p_{\mathrm{c}}(n)> 2$), and an estimate of the $L^{\infty}$ norm via the fractional Sobolev embedding for $n<2\sigma$. Summarizing we may estimate
\begin{align}
	\|u_{\lin}(t,\cdot)\|_{L^{p_{\mathrm{c}}(n)}}&\lesssim (1+t)^{-\frac{1}{p_{\mathrm{c}}(n)}}\|(u_0,u_1)\|_{(H^{\sigma}\cap L^1)\times(L^2\cap L^1)},\label{star-02}\\
	\|u_{\lin}(t,\cdot)\|_{L^{\infty}}&\lesssim (1+t)^{-\frac{n-2\delta}{2(\sigma-\delta)}}\|(u_0,u_1)\|_{(H^{\sigma}\cap L^1)\times(L^2\cap L^1)},\label{star-03}
\end{align}
which imply \eqref{linear-Data-A} immediately.
Thanks to the definition of $\|\cdot\|_{X_T}$, the Riesz-Thorin interpolation between the $L^{p_{\mathrm{c}}(n)}$ and $L^{\infty}$ norms implies 
\begin{align}\label{Star-04}
	\|w(\eta,\cdot)\|_{L^{rp_{\mathrm{c}}(n)}}\lesssim (1+\eta)^{-\frac{1}{2(\sigma-\delta)}\left(n\left(1-\frac{1}{rp_{\mathrm{c}}(n)}\right)-2\delta\right)}\|w\|_{X_{T}}
\end{align}
for any $r\geqslant 1$, with $w=u,\tilde{u}$ and $u-\tilde{u}$ (notice that the power of $(1+\eta)$ is monotone decreasing with respect to $r$).
Moreover, by applying estimates \eqref{star-02}, \eqref{star-03} for $K_1(t-\eta,|D|)$ and recalling $\|u\|_{X_{\eta}},\|\tilde{u}\|_{X_{\eta}}\leqslant \nu$, we may derive
\begin{align*}
	\|\ml{N}[u](t,\cdot)-\ml{N}[\tilde{u}](t,\cdot)\|_{L^q}
	&=\left\|\int_0^tK_1(t-\eta,|D|)\left[g\big(u(\eta,\cdot)\big)-g\big(\tilde{u}(\eta,\cdot)\big)\right]\mathrm{d}\eta\right\|_{L^q}\\
	&\lesssim\int_0^t(1+t-\eta)^{-\frac{1}{2(\sigma-\delta)}\left(n\left(1-\frac{1}{q}\right)-2\delta\right)}\left\|g\big(u(\eta,\cdot)\big)-g\big(\tilde{u}(\eta,\cdot)\big)\right\|_{L^2\cap L^1}\mathrm{d}\eta\\
	&\lesssim \ml{W}_{[u,\tilde{u}],X_T}^{p_{\mathrm{c}}(n)}\int_0^t(1+t-\eta)^{-\frac{1}{2(\sigma-\delta)}\left(n\left(1-\frac{1}{q}\right)-2\delta\right)}(1+\eta)^{-1}\mu\left(2(1+\eta)^{-\frac{n-2\delta}{2(\sigma-\delta)}}\nu\right)\mathrm{d}\eta
\end{align*}
with $q=p_{\mathrm{c}}(n)$ and $q=+\infty$, where we employed Lemma \ref{Lemma-Nonlinearity} with \eqref{Star-04} by taking $r=1,2$ (thanks to our condition $(A_1)$, the assumption \eqref{Assumption-Nonlinearity} holds true taking $\nu$ sufficiently small such that $\tau_1>2\nu$). With a suitable constant $\widetilde{C}_2>0$ one deduces
%
\begin{align*}
	&\int_0^t(1+t-\eta)^{-\alpha_0}(1+\eta)^{-1}\mu\left(\widetilde{C}_2(1+\eta)^{-\alpha_1}\nu\right)\mathrm{d}\eta\notag\\
	&\lesssim (1+t)^{-\alpha_0}\int_0^{t/2}(1+\eta)^{-1}\mu\left(\widetilde{C}_2(1+\eta)^{-\alpha_1}\nu\right)\mathrm{d}\eta+(1+t)^{-1}\mu\left(\widetilde{C}_2(1+t)^{-\alpha_1}\nu/2\right)\int_{t/2}^t(1+t-\eta)^{-\alpha_0}\mathrm{d}\eta\notag\\
	&\lesssim  (1+t)^{-\alpha_0}\int_0^{t/2}(1+\eta)^{-1}\mu\left(\widetilde{C}_2(1+\eta)^{-\alpha_1}\nu\right)\mathrm{d}\eta+(1+t)^{-\alpha_0}\mu\left(\widetilde{C}_2(1+t)^{-\alpha_1}\nu/2\right)
\end{align*}
provided that $\alpha_0< 1$ and $\alpha_1\geqslant 0$, because $t-\eta\sim t$ for $\eta\in [0,t/2]$, $\eta\sim t$ for  $\eta\in[t/2,t]$ and $\mu$ is monotone increasing. Finally, since $\mu$ is continuous and it satisfies the Dini condition \eqref{Dini-Condition} we can estimate
\begin{align*}
	&\int_0^t(1+t-\eta)^{-\frac{1}{2(\sigma-\delta)}\left(n\left(1-\frac{1}{q}\right)-2\delta\right)}(1+\eta)^{-1}\mu\left(2(1+\eta)^{-\frac{n-2\delta}{2(\sigma-\delta)}}\nu\right)\mathrm{d}\eta\lesssim (1+t)^{-\frac{1}{2(\sigma-\delta)}\left(n\left(1-\frac{1}{q}\right)-2\delta\right)},
\end{align*}
for both $q=p_{\mathrm{c}}(n)$ and $q=+\infty$, being 
\begin{align*}
	0\leqslant \frac{n-2\delta}{2(\sigma-\delta)}\leqslant \frac{1}{2(\sigma-\delta)}\left(n\left(1-\frac{1}{p_{\mathrm{c}}(n)}\right)-2\delta\right)= \frac{1}{p_{\mathrm{c}}(n)}<1,
\end{align*}
for any $n>2\delta$ due to \eqref{Critical-Exponent}. Thus, for any $2\delta<n<2\sigma$ the crucial estimate  \eqref{Crucial-02} follows with $C_1(\mu, \nu, T)$ satisfying \eqref{Crucial-04}. 

\paragraph{\large The Limit Case $\delta=\sigma/2$}
According to \cite[Proposition 4.4]{D'Abbicco-Ebert=2017}, we are able to get
\begin{align*}
	\|u_{\lin}(t,\cdot)\|_{L^1}&\lesssim (1+t)\|(u_0,u_1)\|_{L^1\times L^1},\\
	\|u_{\lin}(t,\cdot)\|_{L^{\infty}}&\lesssim (1+t)^{1-\frac{n}{\sigma}}\|(u_0,u_1)\|_{(L^{\infty}\cap L^1)\times (L^{\frac{n}{\sigma}}\cap L^1)},
\end{align*}
where the additional assumption $L^{\infty}\times L^{\frac{n}{\sigma}}$ for the Cauchy data in the $L^{\infty}$ estimate of $u_{\lin}(t,\cdot)$ allows to avoid the singularity as $t\to 0^+$ for large frequencies. By interpolation between $L^1$ and $L^{\infty}$, it is easy to check \eqref{linear-Data-A}. Again from \cite[Proposition 4.4]{D'Abbicco-Ebert=2017} one may claim the following $L^p$-$L^q$ estimate:
\begin{align*}
	\|K_1(t,|D|)f_0\|_{L^q}\lesssim t^{1-\frac{n}{\sigma}\left(\frac{1}{p}-\frac{1}{q}\right)}\|f_0\|_{L^p},
\end{align*}
for $1\leqslant p\leqslant q\leqslant +\infty$. Following  \cite[Proof of Theorem 2.2]{D'Abbicco-Girardi=2023}, for any $t\geqslant t_0>1$ we apply the $L^1$-$L^q$ estimate in $[0,t/2]$ and the $L^q$-$L^q$ estimate in $[t/2,t]$ with $q=p_{\mathrm{c}}(n)$ and $q=+\infty$ to obtain
\begin{align*}
	&\|\ml{N}[u](t,\cdot)-\ml{N}[\tilde{u}](t,\cdot)\|_{L^q}\\
	&\lesssim\int_0^{t/2}(t-\eta)^{1-\frac{n}{\sigma}\left(1-\frac{1}{q}\right)}\left\|g\big(u(\eta,\cdot)\big)-g\big(\tilde{u}(\eta,\cdot)\big)\right\|_{L^1}\mathrm{d}\eta+\int_{t/2}^t(t-\eta)\left\|g\big(u(\eta,\cdot)\big)-g\big(\tilde{u}(\eta,\cdot)\big)\right\|_{L^q}\mathrm{d}\eta\\
	&\lesssim \ml{W}_{[u,\tilde{u}],X_T}^{p_{\mathrm{c}}(n)} \Big[t^{1-\frac{n}{\sigma}\left(1-\frac{1}{q}\right)}\int_0^t(1+\eta)^{-1}\mu\left(C_2(1+\eta)^{1-\frac{n}{\sigma}}\nu\right)\mathrm{d}\eta+(1+t)^{1-\frac{n}{\sigma}\left(1-\frac{1}{q}\right)}\mu(C_4\nu)\Big]
\end{align*}
with a positive constant $C_4>0$ independent of $\nu$; here, for any $n>\sigma$ we employed
\begin{align*}
	&\int_{t/2}^t(t-\eta)(1+\eta)^{-1-\frac{n}{\sigma}\left(1-\frac{1}{q}\right)}\mu\left(2(1+\eta)^{1-\frac{n}{\sigma}}\nu\right)\mathrm{d}\eta\lesssim (1+t)^{-1-\frac{n}{\sigma}\left(1-\frac{1}{q}\right)}\mu(C_4\nu)\int_{t/2}^t(t-\eta)\,\mathrm{d}\eta.
\end{align*}
On the other hand, for any $t\leqslant t_0$, being $1+\eta\gtrsim t-\eta$ if $\eta\in[t/2,t]$, we can employ the $L^q$-$L^q$ estimate with $q=p_{\mathrm{c}}(n)$ and $q=+\infty$ in $[0,t]$ to derive
\begin{align*}
	\|\ml{N}[u](t,\cdot)-\ml{N}[\tilde{u}](t,\cdot)\|_{L^q}
	&\lesssim\ml{W}_{[u,\tilde{u}],X_T}^{p_{\mathrm{c}}(n)}\int_0^t(t-\eta)(1+\eta)^{-1-\frac{n}{\sigma}\left(1-\frac{1}{q}\right)}\mu\left(2(1+\eta)^{1-\frac{n}{\sigma}}\nu\right)\mathrm{d}\eta\\
	&\lesssim \ml{W}_{[u,\tilde{u}],X_T}^{p_{\mathrm{c}}(n)}\left(\int_{t/2}^t(1+\eta)^{-\frac{n}{\sigma}\left(1-\frac{1}{q}\right)}\mu\left(2(1+\eta)^{1-\frac{n}{\sigma}}\nu\right)\mathrm{d}\eta+t^2\mu(2\nu)\right)\\
	&\lesssim \ml{W}_{[u,\tilde{u}],X_T}^{p_{\mathrm{c}}(n)}\left(\int_0^t(1+\eta)^{-1}\mu\left(C_2(1+\eta)^{1-\frac{n}{\sigma}}\nu\right)\mathrm{d}\eta+\mu(C_4\nu)\right).
\end{align*}
The summary of the last two estimates shows \eqref{Crucial-02} with $C_1(\mu,\nu,T)$ satisfying
\begin{align*}
	C_1(\mu,\nu,T)\leqslant \frac{C_3}{2}\left(\int^{C_2\nu}_{C_2(1+T)^{-\frac{n-\sigma}{\sigma}}\nu}\frac{\mu(\tau)}{\tau}\,\mathrm{d}\tau+\mu(C_4\nu)\right).
\end{align*}
Being $p_{\mathrm{c}}(n)>1$ and $\mu(0)=0$, thanks to the continuity of $\mu(\tau)$ as $\tau\to0^+$ we can choose a small constant $\nu$ such that $4C_3\mu(C_4\nu)\leqslant\nu^{1-p_{\mathrm{c}}(n)}$ . Additionally, since $\mu$ satisfies the Dini condition \eqref{Dini-Condition} one can take $\nu$ sufficiently small such that
\begin{align*}
	4C_3\int^{C_2\nu}_{C_2(1+T)^{-\frac{n-\sigma}{\sigma}}\nu}\frac{\mu(\tau)}{\tau}\,\mathrm{d}\tau\leqslant \nu^{1-p_{\mathrm{c}}(n)}.
\end{align*}
Finally, the approach described in Subsection \ref{Subsection-Philosophy} allows to complete  the proof of Theorem \ref{Thm-GESDS} and Theorem \ref{Thm-Lower-Bound}.

\paragraph{\large The Non-effective Case $\delta\in(\sigma/2,\sigma]$}
Let $n=1$ with $\sigma\in(2/3,1)$ or $1<\sigma<n\leqslant \bar{n}(\sigma)$. From \cite[Theorems 2.1 and 9.1]{D'Abbicco-Ebert=2021},  for some $m>n/(2\delta)$ we already know
\begin{align*}
	\|u_{\lin}(t,\cdot)\|_{L^{p_{\mathrm{c}}(n)}}&\lesssim (1+t)^{-\frac{1}{p_{\mathrm{c}}(n)}}\|(u_0,u_1)\|_{(H^{2\delta}_m\cap L^1)\times (L^m\cap L^1)},\\
	\|u_{\lin}(t,\cdot)\|_{L^{\infty}}&\lesssim(1+t)^{1-\frac{n}{\sigma}}\|(u_0,u_1)\|_{(H^{2\delta}_m\cap L^1)\times (L^m\cap L^1)},
\end{align*}
which implies \eqref{linear-Data-A} immediately.
Similarly to the effective case, due to the interpolation between $\|w(\eta,\cdot)\|_{L^{p_{\mathrm{c}}(n)}}$ and $\|w(\eta,\cdot)\|_{L^{\infty}}$ we find
\begin{align}\label{w-2-est}
	\|w(\eta,\cdot)\|_{L^{rp_{\mathrm{c}}(n)}}\lesssim (1+\eta)^{1-\frac{n}{\sigma}\left(1-\frac{1}{rp_{\mathrm{c}}(n)}\right)}\|w\|_{X_{T}}
\end{align}
with $w=u,\tilde{u}$ and $u-\tilde{u}$ for any $r\geqslant 1$, we conclude
\begin{align*}
	\|\ml{N}[u](t,\cdot)-\ml{N}[\tilde{u}](t,\cdot)\|_{L^q}&\lesssim \ml{W}_{[u,\tilde{u}],X_T}^{p_{\mathrm{c}}(n)} \int_0^t(1+t-\eta)^{1-\frac{n}{\sigma}\left(1-\frac{1}{q}\right)}(1+\eta)^{-1}\mu\left(2(1+\eta)^{1-\frac{n}{\sigma}}\nu\right)\mathrm{d}\eta\\
	&\lesssim \ml{W}_{[u,\tilde{u}],X_T}^{p_{\mathrm{c}}(n)} (1+t)^{1-\frac{n}{\sigma}\left(1-\frac{1}{q}\right)}
\end{align*}
with $q=p_{\mathrm{c}}(n)$ and $q=+\infty$, where we used Lemma \ref{Lemma-Nonlinearity} with $\sigma\leqslant n\leqslant 2\sigma$.
In conclusion, we get \eqref{Crucial-02} with $ C_1(\mu,\nu,T)$ satisfying \eqref{Crucial-04}.

\section{Blow-up result and sharp upper bound estimates of the lifespan}\setcounter{equation}{0}\label{Section-Blow-up}
\subsection{Strategy of our proofs}
\label{Section-Blow-up_1}
\hspace{5mm}Analogously to Subsection \ref{Subsection-Philosophy}, for the sake of readability we will explain our approach in proving the blow-up result and the sharp upper bound estimate of the lifespan $T_{\varepsilon}$.

For any $R\geqslant 1$, we will construct a suitable functional $Y(R)\geqslant 0$ which satisfies the following crucial nonlinear differential inequality:
\begin{align}\label{Crucial-05}
	Y'(R)\gtrsim R^{\frac{n}{\kappa
	}}g\left[R^{-\frac{n-\min\{2\delta,\sigma\}}{\kappa}}\big(C_5 Y(R)+C_{u_0,u_1}\varepsilon\big)\right]
\end{align}
with suitable constants $C_5,C_{u_0,u_1}>0$ independent of $R$ and $\varepsilon$, where $C_{u_0,u_1}$ is related to the integral over $\mb{R}^n$ of the initial data $u_0,u_1$.\\
Then, introducing
\begin{align}\label{tilde-YR}
	\widetilde{Y}(R):=C_5Y(R)+C_{u_0,u_1}\varepsilon\geqslant C_{u_0,u_1}\varepsilon>0,
\end{align}
and recalling the definition of the nonlinearity $g$ in \eqref{g-function}, we can derive from \eqref{Crucial-05} 
\begin{align*}
	\frac{\widetilde{Y}'(R)}{[\widetilde{Y}(R)]^{p_{\mathrm{c}}(n)}}&\gtrsim R^{\frac{n-(n-\min\{2\delta,\sigma\})p_{\mathrm{c}}(n)}{\kappa}}\mu\left(R^{-\frac{n-\min\{2\delta,\sigma\}}{\kappa}}\widetilde{Y}(R)\right)\\
	&\gtrsim R^{-1}\mu\left(C_{u_0,u_1}\varepsilon R^{-\frac{n-\min\{2\delta,\sigma\}}{\kappa}}\right),
\end{align*}
thanks to the increasing monotonicity of $\mu$. By integrating the last inequality over $[1,R]$, being
\begin{align*}
	\int_{1}^R\frac{\widetilde{Y}'(\tau)}{[\widetilde{Y}(\tau)]^{p_{\mathrm{c}}(n)}}\,\mathrm{d}\tau\leqslant\frac{1}{p_{\mathrm{c}}(n)-1}\left(C_{u_0,u_1}^{1-p_{\mathrm{c}}(n)}\varepsilon^{1-p_{\mathrm{c}}(n)}-[\widetilde{Y}(R)]^{1-p_{\mathrm{c}}(n)}\right),
\end{align*}
and
\begin{align*}
	\int_{1}^R\tau^{-1}\mu\left(C_{u_0,u_1}\varepsilon \tau^{-\frac{n-\min\{2\delta,\sigma\}}{\kappa}}\right)\mathrm{d}\tau\gtrsim\int_{C_{u_0,u_1}R^{-\frac{n-\min\{2\delta,\sigma\}}{\kappa}}\varepsilon}^{C_{u_0,u_1}\varepsilon}\frac{\mu(\tau)}{\tau}\,\mathrm{d}\tau,
\end{align*}
we derive
\begin{align}\label{Est-widetilde-YR}
	[\widetilde{Y}(R)]^{1-p_{\mathrm{c}}(n)}&\lesssim C_{u_0,u_1}^{1-p_{\mathrm{c}}(n)}\varepsilon^{1-p_{\mathrm{c}}(n)}-\int_{C_{u_0,u_1}R^{-\frac{n-\min\{2\delta,\sigma\}}{\kappa}}\varepsilon}^{C_{u_0,u_1}\varepsilon}\frac{\mu(\tau)}{\tau}\,\mathrm{d}\tau\notag\\
	&\lesssim C_{u_0,u_1}^{1-p_{\mathrm{c}}(n)}\varepsilon^{1-p_{\mathrm{c}}(n)}-\ml{H}\left(C_{u_0,u_1}R^{-\frac{n-\min\{2\delta,\sigma\}}{\kappa}}\varepsilon \right)+\ml{H}(C_{u_0,u_1}\varepsilon)
\end{align}
for any $R\geqslant 1$.
\begin{description}
	\item[(a) Blow-up of Solutions:] Let us assume that the solution $u$ to \eqref{Sigma-Evolution-Modulus} is globally in time defined; then, it turns out that $Y(R)$ and $\widetilde{Y}(R)$ are defined for any $R\geqslant 1$. Nevertheless, this produces a contradiction as $R\to+\infty$; indeed, the left-hand side of \eqref{Est-widetilde-YR} is positive due to \eqref{tilde-YR}, whereas,  by the monotone convergence theorem, the right-hand side tends to $-\ml{H}(0)=-\infty$ since the modulus of continuity $\mu(\tau)$ does not satisfy the Dini condition \eqref{Dini-Condition}. The proof of our Theorem \ref{Thm-Blow-up} follows.
	\item[(b) Upper Bound Estimates of the Lifespan:] Employing again the positivity of $\widetilde{Y}(R)$, we can deduce that the inequality \eqref{Est-widetilde-YR} is violated for any $R\geqslant 1$ such that
	\begin{align}\label{Contradiction-01}
		C_{u_0,u_1}^{1-p_{\mathrm{c}}(n)}\varepsilon^{1-p_{\mathrm{c}}(n)}-\ml{H}\left(C_{u_0,u_1}R^{-\frac{n-\min\{2\delta,\sigma\}}{\kappa}}\varepsilon \right)+\ml{H}(C_{u_0,u_1}\varepsilon)<0;
	\end{align}
	in other words, we obtain that $\widetilde{Y}(R)$ cannot be defined for any $R\geqslant  T_{\varepsilon,u}$, where $T_{\varepsilon,u}$ satisfies the identity 
	\begin{align*}
		\ml{H}\left(C_{u_0,u_1}T_{\varepsilon,u}^{-\frac{n-\min\{2\delta,\sigma\}}{\kappa}}\varepsilon \right)-\ml{H}(C_{u_0,u_1}\varepsilon)=	C_{u_0,u_1}^{1-p_{\mathrm{c}}(n)}\varepsilon^{1-p_{\mathrm{c}}(n)}.
	\end{align*}
	Therefore, defining
	\begin{align*}
		\widetilde{k}_1:=C_{u_0,u_1}^{1-p_{\mathrm{c}}(n)},\ \ \widetilde{k}_2:=C_{u_0,u_1},\ \ \widetilde{K}:=C_{u_0,u_1}^{\frac{\kappa}{n-\min\{2\delta,\sigma\}}},
	\end{align*}
	we derive our desired estimate \eqref{Upper-Bound-Lifespan} in Theorem \ref{Thm-Lower-Bound}.
\end{description}

Finally, let us recall the following generalized version of Jensen's inequality, whose proof also has been shown in \cite[Lemma 8]{Ebert-Girardi-Reissig=2020}.
\begin{lemma}[Jensen's Inequality, \cite{Pick-Kufner-John-Fucik=2013}]\label{Lem-Jensen-Ineq} Let $h(\tau)$ be a convex function on $\mb{R}$. Let $\alpha(x)$ be defined and non-negative almost everywhere on $\Omega$ such that $\alpha(x)$ is positive in a set of positive measure. Then, the following inequality holds:
	\begin{align*}
		h\left(\frac{\int_{\Omega}v(x)\,\alpha(x)\,\mathrm{d}x}{\int_{\Omega}\alpha(x)\,\mathrm{d}x}\right)\leqslant\frac{\int_{\Omega}h(v(x))\,\alpha(x)\,\mathrm{d}x}{\int_{\Omega}\alpha(x)\,\mathrm{d}x}
	\end{align*}
	for all non-negative functions $v(x)$ provided that all the integral terms are meaningful.
\end{lemma}

\subsection{Actions of fractional Laplacians on the general test function}
\hspace{5mm}As a preparation, we first state the following lemma which will be crucial in the choice of the test functions in the proof of our blow-up result and upper bound estimates of the lifespan.
\begin{lemma}\label{Lem-Varphi-Fractional-Derivatives} 
	Let $\bar\sigma\geqslant 0$ and consider
	\begin{align*}
		\Psi(t,x):=\left(1+t^{2\alpha_2}+|x|^{2\beta_2}\right)^{-r_0},
	\end{align*}
	with $\alpha_2\geqslant 1$, $\beta_2\geqslant[\bar{\sigma}]+2$ and $r_0>n/(2\beta_2)$,  in which $[\bar{\sigma}]$ and $\{\bar{\sigma}\}$ denote the integer part and the fractional part of $\bar{\sigma}$, respectively. \\Then, $\Psi\in\ml{C}^2([0,+\infty),\ml{C}^{2+2[\bar{\sigma}]})$; moreover, the following estimates hold for any $x\in\mb{R}^n$ and $t\in[0,1]$:
	\begin{align}
		|\partial_t^j\Psi(t,x)|&\lesssim \Psi(t,x)\ \qquad\qquad\qquad \ \ \mbox{for}\ \ j=0,1,2,\label{Fra-01}\\
		|(-\Delta)^{\bar{\sigma}}\partial_t^j\Psi(t,x)|&\lesssim \left(1+t^{2\alpha_2}+|x|^{2\beta_2}\right)^{-r_1}\ \ \mbox{for}\ \ j=0,1,\label{Fra-02}
	\end{align}
	where 
	$r_1:=
	r_0+\bar\sigma/\beta_2$ if $\bar\sigma\in \mb N_0$, whereas $r_1:=
	(n+2\{\bar\sigma\})/(2\beta_2)$ if $\bar\sigma\in [0,+\infty)\backslash \mb N_0$.
\end{lemma}
\begin{proof}
	Let us begin with calculating
	\begin{align*}
		\partial_t\Psi(t,x)&=A_1\left(1+t^{2\alpha_2}+|x|^{2\beta_2}\right)^{-r_0-1}t^{2\alpha_2-1},\\
		\partial_t^2\Psi(t,x)&=A_2\left(1+t^{2\alpha_2}+|x|^{2\beta_2}\right)^{-r_0-2}t^{2(2\alpha_2-1)}+A_2'\left(1+t^{2\alpha_2}+|x|^{2\beta_2}\right)^{-r_0-1}t^{2(\alpha_2-1)},
	\end{align*}
	with suitable real constants $A_1$, $A_2$, $A_2'$ depending on $\alpha_2$, $r_0$. Thanks to $\alpha_2\geqslant 1$ the estimate \eqref{Fra-01} with $j=0,1,2$ holds true for any $x\in\mb{R}^n$ and $t\in[0,1]$. A direct computation yields
	\begin{align*}
		-\Delta\Psi(t,x)&=-2\beta_2r_0(r_0+1)\left(1+t^{2\alpha_2}+|x|^{2\beta_2}\right)^{-r_0-2}|x|^{4\beta_2-2}\\
		&\quad+2\beta_2(2\beta_2+n-2)r_0\left(1+t^{2\alpha_2}+|x|^{2\beta_2}\right)^{-r_0-1}|x|^{2\beta_2-2},
	\end{align*}
	which iteratively implies
	\begin{align}\label{Rep-Integer-Part}
		(-\Delta)^{[\bar{\sigma}]}\partial_t^j\Psi(t,x)=\sum\limits_{k=1}^{2[\bar{\sigma}]}B_{k,j}\left(1+t^{2\alpha_2}+|x|^{2\beta_2}\right)^{-r_0-k-j}|x|^{2k\beta_2-2[\bar{\sigma}]}t^{(2\alpha_2-1)j},
	\end{align}
	with suitable constants $B_{k,j}$ for $k\in\{1,\dots, 2[\bar{\sigma}]\}$ and $j=0,1$ depending on $\alpha_2$, $\beta_2$, $r_0$, $[\bar{\sigma}]$. \\
	Since $\beta_2\geqslant [\bar{\sigma}]=\bar{\sigma}$, the last equality immediately demonstrates \eqref{Fra-02} for $\bar{\sigma}\in\mb{N}_0$ and $j=0,1$. Similarly, for any $\bar{\sigma}\geqslant 0$ our assumptions $\alpha_2\geqslant 1$ and $\beta_2\geqslant [\bar{\sigma}]+1$ allows to conclude that $\Psi\in\ml{C}^2([0,+\infty),\ml{C}^{2+2[\bar{\sigma}]})$,  avoiding the singularity at $t=0$ and $x=0$.
	
	In the remaining part of this proof we consider the fractional case $\{\bar{\sigma}\}\in(0,1)$; we will separate our discussion into $[\bar{\sigma}]\in\mb{N}_+$ and $[\bar{\sigma}]=0$. \\
	Let us first discuss the case $[\bar{\sigma}]\in\mb{N}_+$. Motivated by the identity \eqref{Rep-Integer-Part}, we introduce the following new auxiliary functions:
	\begin{align*}
		\Psi_{\ell}(t,x):=\left(1+t^{2\alpha_2}+|x|^{2\beta_2}\right)^{-r_0-\ell}|x|^{2\ell\beta_2-2[\bar{\sigma}]}.
	\end{align*}
	For any $\ell\in\{1,\dots,2[\bar{\sigma}]\}$, since $\beta_2\geqslant \bar{\sigma}$ we can claim
	\begin{align}
		\Psi_{\ell}(t,x)&\lesssim |x|^{2\{\bar{\sigma}\}} \left(1+t^{2\alpha_2}+|x|^{2\beta_2}\right)^{-r_0-\frac{\sigma}{\beta_2}}\notag\\
		& \lesssim |x|^{2\{\bar{\sigma}\}}K(t,x) \ \ \mbox{with}\ \  K(t,x): =\left(1+t^{2\alpha_2}+|x|^{2\beta_2}\right)^{-\frac{n+2\{\bar\sigma\}}{2\beta_2}},\label{Est-Psi-ell-Psi}
	\end{align}
	being $r_0>n/(2\beta_2)$. By properly modifying \eqref{Frac-Defn} we can write
	\begin{align*}
		(-\Delta)^{\{\bar{\sigma}\}}\Psi_{\ell}(t,x)=-C_{2\{\bar{\sigma}\}}\int_{\mb{R}^n}\frac{\Psi_{\ell}(t,x+y)-2\Psi_{\ell}(t,x)+\Psi_{\ell}(t,x-y)}{|y|^{n+2\{\bar{\sigma}\}}}\,\mathrm{d}y.
	\end{align*}
	The Taylor theorem implies that
	\begin{align*}
		&\Psi_{\ell}(t,x\pm y)-\Psi_{\ell}(t,x)=\pm \nabla\Psi_{\ell}(t,x)\cdot y+\sum\limits_{|\bar{\alpha}|=2}\frac{|\bar{\alpha}|}{\bar{\alpha}!}(\pm y)^{\bar{\alpha}}\int_0^1(1-\Theta)\,\partial^{\bar{\alpha}}\Psi_{\ell}(t,x\pm \Theta y)\,\mathrm{d}\Theta.
	\end{align*}
	Due to the symmetry with respect to $y$, we may further rewrite the fractional Laplacian acting on the test function $\Psi_{\ell}(t,x)$ as follows:
	\begin{align}\label{Rep-Psi-ell}
		&(-\Delta)^{\{\bar{\sigma}\}}\Psi_{\ell}(t,x)=2C_{2\{\bar{\sigma}\}}\int_{|y|>\bar{r}}\frac{\Psi_{\ell}(t,x)-\Psi_{\ell}(t,x+y)}{|y|^{n+2\{\bar{\sigma}\}}}\,\mathrm{d}y\notag\\
		&\qquad\qquad\qquad\qquad-2C_{2\{\bar{\sigma}\}}\sum\limits_{|\bar{\alpha}|=2}\frac{|\bar{\alpha}|}{\bar{\alpha}!}\int_{|y|<\bar{r}}\frac{y^{\bar{\alpha}}}{|y|^{n+2\{\bar{\sigma}\}}}\int_0^1(1-\Theta)\,\partial^{\bar{\alpha}}\Psi_{\ell}(t,x+\Theta y)\,\mathrm{d}\Theta\,\mathrm{d}y,
	\end{align}
	for any $\bar{r}>0$. Let us take $\bar{r}=|x|/2$ and employ \eqref{Est-Psi-ell-Psi} to get 
	\begin{align*}
		\int_{|y|>|x|/2}\frac{\Psi_{\ell}(t,x)}{|y|^{n+2\{\bar{\sigma}\}}}\,\mathrm{d}y\lesssim K(t,x)|x|^{2\{\bar{\sigma}\}}\int_{|y|>|x|/2}\frac{1}{|y|^{n+2\{\bar{\sigma}\}}}\,\mathrm{d}y\lesssim K(t,x).
	\end{align*}
	Moreover, we may divide the domain $\{y:|y|>|x|/2 \}$ into $\Omega_1:=\{y:2|x|>|y|>|x|/2\}$ and $\Omega_2:=\{y:|y|\geqslant 2|x|\}$. Since $\beta_2\geqslant [\bar{\sigma}]+2$ and $r_0>n/(2\beta_2)$, there exists a small constant $\epsilon>0$ such that $\beta_2\geqslant \bar{\sigma}+\epsilon$ and $r_0>(n+2\epsilon)/(2\beta_2)$ allowing us in the sub-domain $\Omega_2$ to derive
	\begin{align*}
		\Psi_{\ell}(t,x+y)& = \left(1+t^{2\alpha_2}+|x+y|^{2\beta_2}\right)^{-r_0-\frac{\bar{\sigma}+\epsilon}{\beta_2}}\frac{|x+y|^{2\{\bar\sigma\}+2\epsilon}|x+y|^{2\beta_2(\ell-\frac{\bar\sigma+\epsilon}{\beta_2})}}{(1+t^{2\alpha_2}+|x+y|^{2\beta_2})^{\ell-\frac{\bar{\sigma}+\epsilon}{\beta_2}}}\\
		&\lesssim \left(1+t^{2\alpha_2}+|x+y|^{2\beta_2}\right)^{-r_0-\frac{\bar{\sigma}+\epsilon}{\beta_2}}|x+y|^{2\{\bar{\sigma}\}+2\epsilon}\\
		&\lesssim \left(1+t^{2\alpha_2}+|x|^{2\beta_2}\right)^{-r_0-\frac{\bar\sigma}{\beta_2}+\frac{\epsilon}{\beta_2}}\left(1+|y|^{2\beta_2}\right)^{-\frac{2\epsilon}{\beta_2}}|y|^{2\{\bar{\sigma}\}+2\epsilon}\\
		&\lesssim K(t,x)\left(1+|y|^{2\beta_2}\right)^{-\frac{2\epsilon}{\beta_2}}|y|^{2\{\bar{\sigma}\}+2\epsilon},
	\end{align*}
	where we used 
	\begin{align}\label{Triangle-01}
		2|y|>|x+y|\geqslant |y|-|x|>|y|/2>|x|\ \ \mbox{in}\ \ \Omega_2.
	\end{align}
	Since $\Omega_2\subset \mb{R}_+$, the next uniform estimate (with respect to $t,x$) holds:
	\begin{align*}
		\int_{\Omega_2}\frac{\Psi_{\ell}(t,x+y)}{|y|^{n+2\{\bar{\sigma}\}}}\,\mathrm{d}y\lesssim K(t,x)\int_0^{+\infty}|y|^{-1+2\epsilon}\left(1+|y|^{2\beta_2}\right)^{-\frac{2\epsilon}{\beta_2}}\mathrm{d}|y|\lesssim K(t,x).
	\end{align*}
	Let us  consider now the integral over $\Omega_1$.
	\begin{itemize}
		\item When $|x|>1$, since $y\in\Omega_1$,  we have
		\begin{align*}
			\Psi_{\ell}(t,x+y)\lesssim \left(1+t^{2\alpha_2}+|x+y|^{2\beta_2}\right)^{-r_0};
		\end{align*}
		here, we used that $\beta_2\geqslant[\bar{\sigma}]$ for $|x+y|<1$, whereas $|x+y|^{-2[\bar{\sigma}]}\leqslant 1$, otherwise. Then,  by defining $z:=x+y$, we may estimate
		\begin{align*}
			\int_{\Omega_1}\frac{\Psi_{\ell}(t,x+y)}{|y|^{n+2\{\bar{\sigma}\}}}\,\mathrm{d}y&\lesssim |x|^{-n-2\{\bar{\sigma}\}}\int_{0}^{+\infty}\left(1+|z|^{2\beta_2}\right)^{-r_0}|z|^{n-1}\,\mathrm{d}|z|\\
			&\lesssim \left(1+t^{2\alpha_2}+|x|^{2\beta_2}\right)^{-\frac{n+2\{\bar{\sigma}\}}{2\beta_2}}=K(t,x),
		\end{align*}
		uniformly in $t\in [0,1]$ and $x\in \mb R^n$ with $|x|>1$, since $2\beta_2r_0>n$ and $|y|\sim |x|$ in $\Omega_1$.
		\item When $|x|< 1$,  since $y\in\Omega_1$ we have $0\leqslant |x+y|<3|x|$; then, with the aid of
		\begin{align*}
			\Psi_{\ell}(t,x+y)\lesssim\frac{|x|^{2\ell\beta_2-2[\bar{\sigma}]}}{(1+t^{2\alpha_2})^{r_0+\ell}}\lesssim\frac{|x|^{2\ell\beta_2-2[\bar{\sigma}]}}{(1+t^{2\alpha_2}+|x|^{2\beta_2})^{r_0+\ell}}
		\end{align*}
		(note that $1+t^{2\alpha_2}\gtrsim (1+t^{2\alpha_2}+|x|^{2\beta_2})/2$ due to $|x|\leqslant 1$), one obtains
		\begin{align*}
			\int_{\Omega_1}\frac{\Psi_{\ell}(t,x+y)}{|y|^{n+2\{\bar{\sigma}\}}}\,\mathrm{d}y\lesssim \frac{|x|^{2\ell\beta_2-2\bar{\sigma}}}{(1+t^{2\alpha_2}+|x|^{2\beta_2})^{r_0+\ell}}\lesssim K(t,x),
		\end{align*}
		uniformly in $t\in [0,1]$ and $x\in \mb R^n$, thanks to $\beta_2\geqslant [\bar{\sigma}]+2\geqslant \bar{\sigma}$.
	\end{itemize}
	A simple computation addresses, for any $|\bar{\alpha}|=2$,
	\begin{align*}
		\partial^{\bar{\alpha}}\Psi_{\ell}(t,x)=\sum\limits_{j=0}^2\widetilde{B}_{j\ell}(\bar{\alpha})\left(1+t^{2\alpha_2}+|x|^{2\beta_2}\right)^{-r_0-\ell-j}|x|^{2(\ell+j)\beta_2-2[\bar{\sigma}]-2},
	\end{align*}
	for suitable constants $\widetilde{B}_{j\ell}(\bar{\alpha})$. Let us recall
	\begin{align*}
		|x+\Theta y|\leqslant |x|+|y|< 3|x|/2 \ \ \mbox{and}\ \ |x+\Theta y|\geqslant \big| |x|-\Theta|y| \big|> |x|/2,
	\end{align*}
	for any $\Theta\in(0,1)$ and $|y|<|x|/2$.
	As a consequence, we find
	\begin{equation}
		\label{eq:reminder}
		\begin{aligned}
			&\left|\sum\limits_{|\bar{\alpha}|=2}\frac{|\bar{\alpha}|}{\bar{\alpha}!}\int_{|y|<|x|/2}\frac{y^{\bar{\alpha}}}{|y|^{n+2\{\bar{\sigma}\}}}\int_0^1(1-\Theta)\,\partial^{\bar{\alpha}}\Psi_{\ell}(t,x+\Theta y)\,\mathrm{d}\Theta\,\mathrm{d}y  \right|\\
			&\lesssim\sum\limits_{|\bar{\alpha}|=2}\sum\limits_{j=0}^2\frac{2\widetilde{B}_{j\ell}(\bar{\alpha})}{\bar{\alpha}!}\int_{|y|<|x|/2}|y|^{-n-2\{\bar{\sigma}\}+2}\\
			&\quad\qquad\times\int_0^1(1-\Theta)\left(1+t^{2\alpha_2}+|x+\Theta y|^{2\beta_2}\right)^{-r_0-\ell-j}|x+\Theta y|^{2(\ell+j)\beta_2-2[\bar{\sigma}]-2}\,\mathrm{d}\Theta\,\mathrm{d}y\\
			&\lesssim\sum\limits_{j=0}^2\widetilde{B}_{j\ell}(\bar\alpha)\left(1+t^{2\alpha_2}+|x|^{2\beta_2}\right)^{-r_0-\ell-j}|x|^{2(\ell+j)\beta_2-2[\bar{\sigma}]-2}\int_{|y|<|x|/2}|y|^{-n-2\{\bar{\sigma}\}+2}\,\mathrm{d}y\\
			&\lesssim \Psi(t,x)\sum\limits_{j=0}^2\frac{|x|^{2(\ell+j)\beta_2-2\bar{\sigma}}}{(1+t^{2\alpha_2}+|x|^{2\beta_2})^{\ell+j}}\lesssim K(t,x)
		\end{aligned}
	\end{equation}
	uniformly in $t\in [0,1]$ and $x\in \mb R^n$ thanks to $\beta_2\geqslant [\bar{\sigma}]+2\geqslant \bar{\sigma}$ and $\{\bar{\sigma}\}\in(0,1)$.\\ Finally, summarizing all the estimates obtained one derives
	\begin{align}
		\label{eq:useful}
		|(-\Delta)^{\{\bar{\sigma}\}}\Psi_{\ell}(t,x)|\lesssim K(t,x),
	\end{align}
	for any $\ell\in\{1,\dots,2[\bar{\sigma}]\}$. According to the representation \eqref{Rep-Integer-Part}, we obtain
	\begin{align}\label{Est-k}
		|(-\Delta)^{\bar{\sigma}}\Psi(t,x)|\lesssim\sum\limits_{k=1}^{2[\bar{\sigma}]}|(-\Delta)^{\{\bar{\sigma}\}}\Psi_k(t,x)|\lesssim \left(1+t^{2\alpha_2}+|x|^{2\beta_2}\right)^{-\frac{n+2\{\bar\sigma\}}{2\beta_2}}
	\end{align}
	for any $x\in\mb{R}^n$ and any $t\in[0,1]$, which is exactly the desired estimate \eqref{Fra-02} when $j=0$ and $\bar{\sigma}\in [1,+\infty)\backslash \mb{N}_+$. Analogously, the estimate \eqref{Fra-02} with $j=1$ when  $\bar{\sigma}\in [1,+\infty)\backslash \mb{N}_+$ can be proved.
	
	Let us prove now the desired result for $\bar{\sigma}\in(0,1)$. We will first discuss the case $|x|>1$, and then the case $|x|<1$. The representation \eqref{Rep-Psi-ell} with $\bar{r}=|x|/2$ can be employed again for $(-\Delta)^{\bar{\sigma}}\Psi(t,x)$ when $|x|>1$, giving the following identity:
	\begin{align*}
		(-\Delta)^{\bar{\sigma}}\Psi(t,x)&=2C_{2\bar{\sigma}}\int_{|y|>|x|/2}\frac{\Psi(t,x)-\Psi(t,x+y)}{|y|^{n+2\bar{\sigma}}}\,\mathrm{d}y\\
		&\quad-2C_{2\bar{\sigma}}\sum\limits_{|\bar{\alpha}|=2}\frac{|\bar{\alpha}|}{\bar{\alpha}!}\int_{|y|<|x|/2}\frac{y^{\bar{\alpha}}}{|y|^{n+2\bar{\sigma}}}\int_0^1(1-\Theta)\,\partial^{\bar{\alpha}}\Psi(t,x+\Theta y)\,\mathrm{d}\Theta\,\mathrm{d}y.
	\end{align*}
	On the one hand, due to $|x|>1$ one notices
	\begin{align*}
		\int_{|y|>|x|/2}\frac{\Psi(t,x)}{|y|^{n+2\bar{\sigma}}}\,\mathrm{d}y\lesssim\Psi(t,x)|x|^{-2\bar{\sigma}}\lesssim \left(1+t^{2\alpha_2}+|x|^{2\beta_2}\right)^{-r_0-\frac{\bar\sigma}{\beta_2}}\lesssim K(t,x).
	\end{align*}
	On the other hand, recalling the definitions of $\Omega_1$ and $\Omega_2$, we can employ the same approach used to treat $\Psi_{\ell}(t,x+y)$ (especially \eqref{Triangle-01} for $y\in\Omega_2$) to estimate
	\begin{align*}
		\int_{|y|>|x|/2}\frac{\Psi(t,x+y)}{|y|^{n+2\bar{\sigma}}}\,\mathrm{d}y&=\left(\int_{\Omega_1}+\int_{\Omega_2}\right)\frac{\Psi(t,x+y)}{|y|^{n+2\bar{\sigma}}}\,\mathrm{d}y\\
		&\lesssim|x|^{-n-2\bar{\sigma}}\int_0^{+\infty}\left(1+|z|^{2\beta_2}\right)^{-r_0}|z|^{n-1}\,\mathrm{d}|z|+\Psi(t,x)\int_{2|x|}^{+\infty}|y|^{-2\bar{\sigma}-1}\,\mathrm{d}|y|\\
		&\lesssim |x|^{-n-2\bar{\sigma}}+\Psi(t,x)|x|^{-2\bar{\sigma}}\lesssim K(t,x),
	\end{align*}
	uniformly in $t\in [0,1]$ and $x\in \mb R^n$ for $|x|>1$,  thanks to $2\beta_2r_0>n$. Moreover, we know that for any $|\bar{\alpha}|=2$ it holds 
	\begin{align*}
		\partial^{\bar{\alpha}}\Psi(t,x)=\sum\limits_{j=1}^2\widetilde{D}_{j}(\bar{\alpha})\left(1+t^{2\alpha_2}+|x|^{2\beta_2}\right)^{-r_0-j}|x|^{2j\beta_2-2},
	\end{align*}
	for suitable constants $\widetilde{D}_{j}(\bar{\alpha})$. Thus, proceeding as in \eqref{eq:reminder} one can prove
	\begin{align*}
		\left|\sum\limits_{|\bar{\alpha}|=2}\frac{|\bar{\alpha}|}{\bar{\alpha}!}\int_{|y|<|x|/2}\frac{y^{\bar{\alpha}}}{|y|^{n+2\bar{\sigma}}}\int_0^1(1-\Theta)\,\partial^{\bar{\alpha}}\Psi(t,x+\Theta y)\,\mathrm{d}\Theta\,\mathrm{d}y \right|\lesssim \Psi(t,x),
	\end{align*}
	when $|x|>1$, being $\bar{\sigma}\in(0,1)$. \\
	Summarizing the previous estimates we obtain that \eqref{Fra-02} holds true for any $|x|>1$ and $t\in[0,1]$, when $j=0$ and $\bar\sigma\in (0,1)$. 
	
	In order to prove \eqref{Fra-02} for $|x|<1$ we apply Taylor's formula: for any fixed $|x|<1$ there exists $\xi_x\in \mb R^n$ with $|\xi_x|<|x|$ such that  
	\begin{equation}
		\label{eq:Taylor_|x|<1}
		(-\Delta)^{\bar\sigma} \Psi(t,x)=(-\Delta)^{\bar\sigma} \Psi(t,0)+ \nabla_x (-\Delta)^{\bar\sigma} \Psi(t,\xi_x)\cdot x.
	\end{equation}
	Being $\beta\geqslant 2$ we can easily deduce that $\Psi\in \ml{C}^2([0,+\infty), \ml{C}^2(\mb R^n))$; moreover, it holds
	\begin{align*}
		\nabla_x (-\Delta)^{\bar\sigma}  \Psi(t,\xi_x)\cdot x = \sum_{i=1}^n  \big((-\Delta)^{\bar\sigma}  \partial_{x_j}\Psi\big)(t,\xi_x) x_j.
	\end{align*}
	Applying the same procedure used to prove estimate \eqref{eq:useful} one can derive
	\begin{align*}
		\left|\big((-\Delta)^{\bar\sigma} \partial_{x_j}\Psi\big)(t,\xi_x) x_j\right|\lesssim   K(t,x)|x|\lesssim 1,
	\end{align*}
	for any $|x|<1$ and $t\in [0,1]$. Thus, from the identity \eqref{eq:Taylor_|x|<1} we obtain 
	\begin{equation}
		\label{eq:laplace_phi_|x|<1}
		|(-\Delta)^{\bar\sigma}  \Psi(t,x)|\lesssim 1,
	\end{equation}
	for any $|x|<1$ and any $t\in[0,1]$, which completes the proof of \eqref{Fra-02} with $j=0$ when $[\bar{\sigma}]=0$.  Similarly, the estimate \eqref{Fra-02} with $j=1$ when $[\bar{\sigma}]=0$ can be proved.
\end{proof}

\subsection{Proofs of Theorems \ref{Thm-Blow-up} and \ref{Thm-Upper-Bound}}
\hspace{5mm}Let us consider
\begin{align*}
	\varphi(t,x):=\left(1+t^{2\theta}+|x|^{2\theta\kappa}\right)^{-\frac{q_0}{2\theta\kappa}},
\end{align*}
with $\theta\geqslant\max\{1,([\sigma]+1)/\kappa \}$, and $\rho\in\ml{C}^2([0,+\infty))$ such that
\begin{align*}
	\rho(t):=\begin{cases}
		1&\mbox{if}\ \ t\in[0,1/2],\\
		\mbox{non-increasing}&\mbox{if}\ \ t\in[1/2,1],\\
		0&\mbox{if}\ \ t\in[1,+\infty),
	\end{cases}
\end{align*}
where $q_0=n+2s_0$ with $s_0$ defined by \eqref{eq:s0}. 

Applying Lemma \ref{Lem-Varphi-Fractional-Derivatives} with $\alpha_2=\theta$, $\beta_2=\theta \kappa$ and $r_0=q_0/(2\theta\kappa)$ it is easy to check that $\varphi$ belongs to $\mathfrak{X}_{q_0}([0,T)\times\mb{R}^n)$ according to Definition \ref{Defn-test-space}. 
In particular, due to our choice of $r_0$ we may estimate
\begin{align}
	|\partial_t^2\varphi(t,x)|, \; |(-\Delta)^{\sigma} \varphi(t,x)|,\; |(-\Delta)^{\delta}\partial_t\varphi(t,x)| \lesssim  \varphi(t,x), \label{eq:phi_estimate}
\end{align}
uniformly with respect to $(t,x)\in [0,1]\times \mb R^n$. \\
Then, we introduce $\psi(t,x):=[\rho(t)]^{r_2}\varphi(t,x)$ with $r_2\geqslant 2p'_{\mathrm{c}}(n)$; moreover, for any $R>0$ we define
\begin{align*}
	\psi_R(t,x):=\psi\left(R^{-1}t,R^{-\frac{1}{\kappa}}x\right).
\end{align*}
Assume by contradiction that $u\in L_{\mathrm{loc}}^{p_{\mathrm{c}}(n)}\left([0,+\infty), L^{p_{\mathrm{c}}(n)}(\mb{R}^n;\langle x\rangle^{-q_0}\,\mathrm{d}x)\right)$ is a global in time weak solution to the semilinear Cauchy problem \eqref{Sigma-Evolution-Modulus} in the sense of Definition \ref{Defn-Weak-Solution}. Thus,  for any $R\geqslant 1$ we can consider the following integral term:
\begin{align*}
	I_R:=\int_0^{+\infty}\int_{\mb{R}^n}g\big(u(t,x)\big)\psi_R(t,x)\,\mathrm{d}x\,\mathrm{d}t,
\end{align*}
where $g(u)=|u|^{p_{\mathrm{c}}(n)}\mu(|u|)$; according to \eqref{Equality-Weak-Solution} the integral term $I_R$ satisfies the following identity:
\begin{align*}
	I_R+C_{u_0,u_1}^{(R)}\varepsilon=\int_0^{+\infty}\int_{\mb{R}^n}u(t,x)\left(\partial_t^2+(-\Delta)^{\sigma}-(-\Delta)^{\delta}\partial_t\right)\psi_R(t,x)\,\mathrm{d}x\,\mathrm{d}t;
\end{align*}
here,  the constant $C^{(R)}_{u_0,u_1}$, depending on $u_0,u_1$ and $R$, is defined by:
\begin{align*}
	C_{u_0,u_1}^{(R)}:=\int_{\mb{R}^n}\left[u_1(x)\,\psi_R(0,x)-u_0(x)\left(\partial_t\psi_R(0,x)-(-\Delta)^{\delta}\psi_R(0,x)\right) \right]\mathrm{d}x.
\end{align*}
Employing \eqref{eq:phi_estimate} we may estimate
\begin{align*}
	|(-\Delta)^{\sigma}\psi_R(t,x)|&=R^{-\frac{2\sigma}{\kappa}}[\rho(R^{-1}t)]^{r_2}\left|\big((-\Delta)^{\sigma}\varphi\big)\left(R^{-1}t,R^{-\frac{1}{\kappa}}x\right)\right|\\
	&\lesssim R^{-\frac{2\sigma}{\kappa}}[\rho(R^{-1}t)]^{r_2}\varphi_R(t,x).
\end{align*}
Due to the fact that
\begin{align*}
	\partial_t\psi(t,x)&=r_2[\rho(t)]^{r_2-1}\rho'(t)\,\varphi(t,x)+[\rho(t)]^{r_2}\partial_t\varphi(t,x),
\end{align*}
with the aid of $\rho\in\ml{C}^2_0([0,+\infty))$ and the estimate \eqref{eq:phi_estimate} we may derive
\begin{align*}
	|(-\Delta)^{\delta}\partial_t\psi_R(t,x)|&=R^{-1-\frac{2\delta}{\kappa}}\left|\big((-\Delta)^{\delta}\partial_t\psi\big)\left(R^{-1}t,R^{-\frac{1}{\kappa}}x\right)\right|\\
	&\lesssim R^{-1-\frac{2\delta}{\kappa}}[\rho(R^{-1}t)]^{r_2-1}\varphi_R(t,x),
\end{align*}
for any $\delta\in [0,\sigma]$. Similarly, it holds
\begin{align*}
	|\partial_t^2\psi_R(t,x)|\lesssim R^{-2}\left|\partial^2_t\psi\left(R^{-1}t,R^{-\frac{1}{\kappa}}x\right)\right|\lesssim R^{-2}[\rho(R^{-1}t)]^{r_2-2}\varphi_R(t,x).
\end{align*}
Summarizing the last derived estimates, we conclude immediately
\begin{equation}
	\begin{aligned}
		\label{eq:main_IR}
		I_R+C_{u_0,u_1}^{(R)}\varepsilon  \lesssim R^{-\min\left\{\frac{2\sigma}{\kappa},1+\frac{2\delta}{\kappa},2 \right\}}\int_0^{+\infty}\int_{\mb{R}^n}|u(t,x)|\,[\rho(R^{-1}t)]^{r_2-2}\varphi_R(t,x)\,\mathrm{d}x\,\mathrm{d}t;
	\end{aligned}
\end{equation}
in particular, we notice that $\min\{2\sigma/\kappa,1+2\delta/\kappa,2 \}=2\sigma/\kappa$ for all $\delta\in [0,\sigma]$.\\
In order to derive sharp upper bound estimates of the lifespan, we introduce a new auxiliary function
\begin{align*}
	\Phi(t,x):=\frac{(t^{2\theta}+|x|^{2\theta\kappa})^{\beta_0}}{(1+t^{2\theta}+|x|^{2\theta\kappa})^{\beta_1}},
\end{align*}
with $0<\beta_0<\beta_1$ such that
\begin{align}\label{beta-0-1}
	\beta_0<\frac{p_{\mathrm{c}}(n)-1}{2\theta}\ \ \mbox{and}\ \ \beta_1<\frac{s_0(p_{\mathrm{c}
		}(n)-1)}{\theta\kappa}.
\end{align}
Moreover, we define
\begin{align*}
	\Phi_R(t,x):=\Phi\left(R^{-1}t,R^{-\frac{1}{\kappa}}x\right).
\end{align*}
Thanks to the convexity of $g$ in the assumption $(A_3)$, we can apply the generalized Jensen's inequality, i.e. Lemma \ref{Lem-Jensen-Ineq}, with the non-negative functions
\begin{align*}
	v(t,x)&:=|u(t,x)|\,[\rho(R^{-1}t)]^{\frac{2}{p_{\mathrm{c}}(n)-1}}[\Phi_R(t,x)]^{\frac{1}{p_{\mathrm{c}}(n)-1}},\\
	\alpha(t,x)&:=[\rho(R^{-1}t)]^{r_2-2p'_{\mathrm{c}}(n)}[\Phi_R(t,x)]^{-\frac{1}{p_{\mathrm{c}}(n)-1}}\varphi_R(t,x),
\end{align*}
to arrive at
\begin{align}\label{G-Jensen}
	&g\left(\frac{\int_0^{+\infty}\int_{\mb{R}^n}|u(t,x)|\,[\rho(R^{-1}t)]^{r_2-2}\varphi_R(t,x)\,\mathrm{d}x\,\mathrm{d}t}{\int_0^{+\infty}\int_{\mb{R}^n}\alpha(t,x)\,\mathrm{d}x\,\mathrm{d}t} \right)\leqslant\frac{\int_0^{+\infty}\int_{\mb{R}^n}g(v(t,x))\,\alpha(t,x)\,\mathrm{d}x\,\mathrm{d}t}{\int_0^{+\infty}\int_{\mb{R}^n}\alpha(t,x)\,\mathrm{d}x\,\mathrm{d}t}.
\end{align}
Recalling the definition of the nonlinearity $g$, one may estimate
\begin{align*}
	g\big(v(t,x)\big)\,\alpha(t,x)
	&=|u(t,x)|^{p_{\mathrm{c}}(n)}\mu\left(|u(t,x)|\,[\rho(R^{-1}t)]^{\frac{2}{p_{\mathrm{c}}(n)-1}}[\Phi_R(t,x)]^{\frac{1}{p_{\mathrm{c}}(n)-1}}\right)\Phi_R(t,x)\,\psi_R(t,x)\\
	&\leqslant g\big(u(t,x)\big)\,\Phi_R(t,x)\,\psi_R(t,x),
\end{align*}
where we used the monotone increasing property of $\mu$ combined with the boundedness of $0\leqslant \rho(t),\Phi_R(t,x)\leqslant 1$. Moreover, applying the changes of variables $t\mapsto R\,t$ and $x\mapsto R^{\frac{1}{\kappa}}x$, we deduce
\begin{align*}
	\int_0^{+\infty}\int_{\mb{R}^n}\alpha(t,x)\,\mathrm{d}x\,\mathrm{d}t&=\int_0^{+\infty}\int_{\mb{R}^n}[\rho(R^{-1}t)]^{r_2-2p'_{\mathrm{c}}(n)}[\Phi_R(t,x)]^{-\frac{1}{p_{\mathrm{c}}(n)-1}}\varphi_R(t,x)\,\mathrm{d}x\,\mathrm{d}t\\
	&=\int_0^{+\infty}\int_{\mb{R}^n}[\rho(t)]^{r_2-2p'_{\mathrm{c}}(n)}[\Phi(t,x)]^{-\frac{1}{p_{\mathrm{c}}(n)-1}}\varphi(t,x)\,\mathrm{d}x\,\mathrm{d}t\,R^{1+\frac{n}{\kappa}}\\
	&=:C_{\mathrm{int}}\,R^{1+\frac{n}{\kappa}},
\end{align*}
where $C_{\mathrm{int}}$ is a positive and bounded constant; indeed, being $r_2\geqslant 2 p'_c(n)$ and $0\leqslant \rho, \varphi \leqslant 1$ we have
\begin{align*}
	0<C_{\mathrm{int}}&\leqslant\int_0^1\int_{\mb{R}^n}\frac{(t^{2\theta}+|x|^{2\theta\kappa})^{-\frac{\beta_0}{p_{\mathrm{c}}(n)-1}}}{(1+t^{2\theta}+|x|^{2\theta\kappa})^{-\frac{\beta_1}{p_{\mathrm{c}}(n)-1}+\frac{q_0}{2\theta\kappa}}}\,\mathrm{d}x\,\mathrm{d}t\\
	&\lesssim\left(\int_0^1t^{-\frac{2\theta\beta_0}{p_{\mathrm{c}}(n)-1}}\mathrm{d}t\right)\left(\int_{\mb{R}^n}(1+|x|^{2\theta\kappa})^{\frac{\beta_1}{p_{\mathrm{c}}(n)-1}-\frac{q_0}{2\theta\kappa}}\mathrm{d}x\right)\lesssim 1,
\end{align*}
thanks to $r_2\geqslant 2p'_{\mathrm{c}}(n)$ and the restrictions on $\beta_0$, $\beta_1$ in \eqref{beta-0-1}. From the monotone increasing property of $g$, \eqref{G-Jensen} implies that
\begin{align*}
	&\int_0^{+\infty}\int_{\mb{R}^n}|u(t,x)|\,[\rho(R^{-1}t)]^{r_2-2}\varphi_R(t,x)\,\mathrm{d}x\,\mathrm{d}t\\
	&\qquad\lesssim R^{1+\frac{n}{\kappa}}g^{-1}\left(C_{\mathrm{int}}^{-1}R^{-1-\frac{n}{\kappa}}\int_0^{+\infty}\int_{\mb{R}^n}g\big(u(t,x)\big)\,\Phi_R(t,x)\,\psi_R(t,x)\,\mathrm{d}x\,\mathrm{d}t\right).
\end{align*}
Thus, by \eqref{eq:main_IR} the following estimate follows:
\begin{align*}
	&I_R+C_{u_0,u_1}^{(R)}\varepsilon\lesssim R^{1+\frac{n-2\sigma}{\kappa}}g^{-1}\left(C_{\mathrm{int}}^{-1}R^{-1-\frac{n}{\kappa}}\int_0^{+\infty}\int_{\mb{R}^n}g\big(u(t,x)\big)\,\Phi_R(t,x)\,\psi_R(t,x)\,\mathrm{d}x\,\mathrm{d}t\right). 
\end{align*}

Let us introduce an auxiliary functional
\begin{align*}
	Y(R):=\int_0^Ry(r)\,r^{-1}\,\mathrm{d}r\ \ \mbox{with}\ \ y(r):=\int_0^{+\infty}\int_{\mb{R}^n}g\big(u(t,x)\big)\,\Phi_r(t,x)\,\psi_r(t,x)\,\mathrm{d}x\,\mathrm{d}t.
\end{align*}
By the identity $y(R)=Y'(R)R$, we may rewrite the last inequality by
\begin{equation}
	\label{eq:IR leqY'(R)}
	I_R+C_{u_0,u_1}^{(R)}\varepsilon\lesssim R^{1+\frac{n-2\sigma}{\kappa}}g^{-1}\left(C_{\mathrm{int}}^{-1}R^{-\frac{n}{\kappa}}Y'(R)\right).
\end{equation}
Moreover, due to $\psi_r\leqslant \psi_R$ for any $r\in(0,R]$, we estimate
\begin{equation}
	\label{eq:Y(R) leq IR}
	\begin{aligned}
		Y(R)&\lesssim\int_0^{+\infty}\int_{\mb{R}^n}g\big(u(t,x)\big)\,\psi_R(t,x)\,\left(\int_0^R\Phi_r(t,x)\,r^{-1}\,\mathrm{d}r\right)\mathrm{d}x\,\mathrm{d}t\\
		&\lesssim\int_0^{+\infty}\int_{\mb{R}^n}g\big(u(t,x)\big)\,\psi_R(t,x)\,\mathrm{d}x\,\mathrm{d}t=I_R.
	\end{aligned}
\end{equation}
Indeed, employing the change of variable $\tau=r^{-2\theta}(t^{2\theta}+|x|^{2\theta\kappa})$ one can find
\begin{align*}
	\int_0^R\Phi_r(t,x)\,r^{-1}\,\mathrm{d}r=\frac{1}{2\theta}\int_{\frac{t^{2\theta}+|x|^{2\theta\kappa}}{R^{2\theta}}}^{+\infty}\frac{\tau^{\beta_0-1}}{(1+\tau)^{\beta_1}}\,\mathrm{d}\tau\lesssim \int_0^{+\infty}\frac{\tau^{\beta_0-1}}{(1+\tau)^{\beta_1}}\,\mathrm{d}\tau,
\end{align*}
which is uniformly bounded with respect to $R$, being $0<\beta_0<\beta_1$. In conclusion, by \eqref{eq:IR leqY'(R)} and \eqref{eq:Y(R) leq IR} we obtain
\begin{align}\label{Crucial-05bar}
	Y'(R)\gtrsim R^{\frac{n}{\kappa
	}}g\left[R^{-\frac{n-\min\{2\delta,\sigma\}}{\kappa}}\big(C_5 Y(R)+C_{u_0,u_1}^{(R)}\varepsilon\big)\right]
\end{align}
for any $R\geqslant 1$ with a suitable constant $C_5>0$ independent of $R$ and $\varepsilon$.

In particular, since $\rho'(0)=0$ and $\partial_t\varphi(0,x)=0$ due to $\theta\geqslant 1$, we can write 
\begin{align*}
	C_{u_0,u_1}^{(R)}=\int_{\mb{R}^n}\left[u_1(x)\,\psi\left(0,R^{-\frac{1}{\kappa}}x\right)+R^{-\frac{2\delta}{\kappa}}u_0(x)\,(-\Delta)^{\delta}\psi\left(0,R^{-\frac{1}{\kappa}}x\right) \right]\mathrm{d}x;
\end{align*}
if $\delta=0$, being $\psi\big(0, R^{-\frac{1}{\kappa}}x\big)\geqslant \langle x \rangle^{-q_0}$ for any $R\geqslant 1$, we may estimate
\begin{equation}
	\label{eq:Cu0u1_delta=0}
	C_{u_0,u_1}^{(R)}\geqslant \int_{\mb R^n} \big(u_0(x)+u_1(x)\big)\langle x\rangle^{-q_0}\,\mathrm{d}x =:C_{u_0,u_1}>0
\end{equation}
uniformly for any $R\geqslant 1$ thanks to the sign assumption \eqref{eq:data_sign_delta=0}. \\ On the other hand, if $\delta>0$ we may estimate
\begin{align*}
	C_{u_0,u_1}^{(R)}&\geqslant \int_{\mb{R}^n}\left(u_1(x)-c_\delta R^{-\frac{2\delta}{\kappa}}|u_0(x)|\right)\psi\left(0,R^{-\frac{1}{\kappa}}x\right) \mathrm{d}x,
\end{align*}
where $c_\delta$ is a positive constant such that $|(-\Delta)^\delta \psi| \leqslant c_\delta \psi$.\\
We notice that, if assumption \eqref{eq:data_sign_delta>0_1} holds, since $u_0\in L^1(\mb R^n)$ and $\psi\leqslant 1$ the application of the dominated convergence theorem allows to get:
\[ \lim_{R\to +\infty} R^{-\frac{2\delta}{\kappa}}\int_{\mb R^n} |u_0(x)|\psi(0, R^{-\frac{1}{\kappa}})\, \mathrm{d}x=0;\]
in particular, there exists $R_\delta>0$ such that 
\[  c_\delta R^{-\frac{2\delta}{\kappa}}\int_{\mb R^n} |u_0(x)|\psi(0, R^{-\frac{1}{\kappa}})\, \mathrm{d}x<\frac{1}{2} \int_{\mb{R}^n}u_1(x)\langle x \rangle^{-q_0}\, \mathrm{d}x,\]
for any $R\geqslant R_\delta$. Then, $C_{u_0,u_1}^{(R)}$ satisfies
\begin{align}
	\label{eq:Cu0u1_delta>0}
	C_{u_0,u_1}^{(R)}\geqslant\frac{1}{2} \int_{\mb{R}^n}u_1(x) \langle x \rangle^{-q_0}\, \mathrm{d}x=:C_{u_0,u_1}>0,
\end{align}
uniformly with respect to $R\geqslant R_\delta$.  
Similarly, if the initial data satisfy the assumption \eqref{eq:data_sign_delta>0_2} then there exists $\bar c>0$ such that $u_1(x)\geqslant 2\bar{c}|u_0(x)|$ for any $x\in \mb R^n$; in this case, we may fix $R_\delta>0$ sufficiently large such that $c_\delta R^{-\frac{2\delta}{\kappa}}\leqslant \bar{c}$ for any $R\geqslant R_\delta$. Thus, the inequality \eqref{eq:Cu0u1_delta>0} still holds true for any $R\geqslant R_\delta$.\\
Theorem \ref{Thm-Lower-Bound} provides a lower bound $T_{\varepsilon,\ell}$ for $T_\varepsilon$ which allows us to deduce that $T_\varepsilon\to +\infty$ as $\varepsilon\to 0^+$; thus, taking $\bar\varepsilon>0$ sufficiently small one obtains $T_{\varepsilon}>\max\{R_\delta,1\}$ for any $\varepsilon\in (0, \bar\varepsilon)$.
Then, we can claim that the inequality \eqref{Crucial-05bar} holds for any $R\in [R_\delta, T_\varepsilon)$ with $C^{(R)}_{u_0,u_1}$ satisfying \eqref{eq:Cu0u1_delta=0} if $\delta=0$, and \eqref{eq:Cu0u1_delta>0} if $\delta>0$.
In particular, \eqref{Crucial-05bar} implies \eqref{Crucial-05} with $C_{u_0,u_1}$ defined by \eqref{eq:Cu0u1_delta=0} if $\delta=0$, and by \eqref{eq:Cu0u1_delta>0} if $\delta>0$. Then, the strategy explained in Section \ref{Section-Blow-up_1}\,{(a)} allows to conclude the proof of Theorem \ref{Thm-Blow-up} if $T_\varepsilon=+\infty$; on the other hand, if $T_\varepsilon<+\infty$ the approach described in Section \ref{Section-Blow-up_1}\,{(b)} allows to get the desired upper bound \eqref{Upper-Bound-Lifespan} for the lifespan of  solution, provided that $\varepsilon \in (0, \bar\varepsilon)$. The proof of Theorem \ref{Thm-Upper-Bound} follows.

\section*{Acknowledgments}
Wenhui Chen is supported in part by the National Natural Science Foundation of China (grant No. 12301270, grant No. 12171317), Guangdong Basic and Applied Basic Research Foundation (grant No. 2023A1515012044), 2024 Basic and Applied Basic Research Topic--Young Doctor Set Sail Project (grant No. 2024A04J0016). Giovanni Girardi is supported in part by the INdAM-GNAMPA Project (grant No. CUP E53C23001670001).

\end{document}